\documentclass[11pt,a4paper,reqno]{amsart}
\usepackage{amsfonts,amsthm,amscd,epsfig,amsmath,amssymb,enumerate}
\numberwithin{equation}{section}

\DeclareMathSymbol{\leqslant}{\mathalpha}{AMSa}{"36} % nicer `smaller or equal'
\DeclareMathSymbol{\geqslant}{\mathalpha}{AMSa}{"3E} % nicer `larger or equal'
\DeclareMathSymbol{\eset}{\mathalpha}{AMSb}{"3F}     % nicer `emptyset'
\renewcommand{\leq}{\;\leqslant\;}                   % redef. of < or =
\renewcommand{\geq}{\;\geqslant\;}                   % redef. of > or =
\newcommand{\be}{\begin{equation}}

\def\1{\ifmmode {1\hskip -3pt \rm{I}} \else {\hbox {$1\hskip -3pt \rm{I}$}}\fi}

\newtheorem{Th}{Theorem}[section]
\newtheorem{Le}[Th]{Lemma}
\newtheorem{Pro}[Th]{Proposition}
\newtheorem{Cor}[Th]{Corollary}

\newcommand{\cA}{\ensuremath{\mathcal A}}
\newcommand{\cB}{\ensuremath{\mathcal B}}
\newcommand{\cC}{\ensuremath{\mathcal C}}
\newcommand{\cD}{\ensuremath{\mathcal D}}

\newcommand{\cF}{\ensuremath{\mathcal F}}

\newcommand{\cH}{\ensuremath{\mathcal H}}

\newcommand{\cK}{\ensuremath{\mathcal K}}
\newcommand{\cL}{\ensuremath{\mathcal L}}

\newcommand{\cN}{\ensuremath{\mathcal N}}

\newcommand{\cP}{\ensuremath{\mathcal P}}

\newcommand{\cT}{\ensuremath{\mathcal T}}

\newcommand{\cV}{\ensuremath{\mathcal V}}
\newcommand{\cW}{\ensuremath{\mathcal W}}

\newcommand{\bbC}{{\ensuremath{\mathbb C}} }

\newcommand{\bbE}{{\ensuremath{\mathbb E}} }

\newcommand{\bbI}{{\ensuremath{\mathbb I}} }

\newcommand{\bbL}{{\ensuremath{\mathbb L}} }

\newcommand{\bbN}{{\ensuremath{\mathbb N}} }

\newcommand{\bbP}{{\ensuremath{\mathbb P}} }
\newcommand{\bbQ}{{\ensuremath{\mathbb Q}} }
\newcommand{\bbR}{{\ensuremath{\mathbb R}} }

\newcommand{\bbZ}{{\ensuremath{\mathbb Z}} }

\newcommand{\wt}{\widetilde}

\let\a=\alpha \let\b=\beta   \let\d=\delta  \let\e=\varepsilon
 \let\g=\gamma       \let\l=\lambda
          \let\p=\pi  
  \let\s=\sigma \let\t=\tau   
  
\let\D=\Delta   \let\G=\Gamma

\def\\{\hfill\break}

\def\tthsp{\kern .083333 em}

\def\?{\mskip -10mu}
\def\indbox#1{\hbox to \parindent{\hfil\ #1\hfil} }

\newfam\msafam
\newfam\msbfam
\newfam\eufmfam
\def\hexnumber#1{%
  \ifcase#1 0\or 1\or 2\or 3\or 4\or 5\or 6\or 7\or 8\or
  9\or A\or B\or C\or D\or E\or F\fi}
\font\tenmsa=msam10 \font\sevenmsa=msam7 \font\fivemsa=msam5
\textfont\msafam=\tenmsa \scriptfont\msafam=\sevenmsa
\scriptscriptfont\msafam=\fivemsa
\edef\msafamhexnumber{\hexnumber\msafam}%
\mathchardef\restriction"1\msafamhexnumber16 \mathchardef\ssim"0218
\mathchardef\square"0\msafamhexnumber03
\mathchardef\eqd"3\msafamhexnumber2C
\def\QED{\ifhmode\unskip\nobreak\fi\quad
  \ifmmode\square\else$\square$\fi}
\font\tenmsb=msbm10 \font\sevenmsb=msbm7 \font\fivemsb=msbm5
\textfont\msbfam=\tenmsb \scriptfont\msbfam=\sevenmsb
\scriptscriptfont\msbfam=\fivemsb

\font\teneufm=eufm10 \font\seveneufm=eufm7 \font\fiveeufm=eufm5
\textfont\eufmfam=\teneufm \scriptfont\eufmfam=\seveneufm
\scriptscriptfont\eufmfam=\fiveeufm

\def\({\left(}
\def\){\right)}
\let\neper=e
\let\ii=i

\let\<=\langle
\let\>=\rangle

\outer\def\nproclaim#1 [#2]#3. #4\par{\medbreak \noindent
   \talato(#2){\bf #1 \Thm[#2]#3.\enspace }%
   {\sl #4\par }\ifdim \lastskip <\medskipamount
   \removelastskip \penalty 55\medskip \fi}
\def\thmm[#1]{#1}
\def\teo[#1]{#1}
\def\sttilde#1{%
\dimen2=\fontdimen5\textfont0 \setbox0=\hbox{$\mathchar"7E$}
\setbox1=\hbox{$\scriptstyle #1$} \dimen0=\wd0 \dimen1=\wd1
\advance\dimen1 by -\dimen0 \divide\dimen1 by 2
\vbox{\offinterlineskip%
   \moveright\dimen1 \box0 \kern - \dimen2\box1}
}
 \title[Spectral analysis of 1D nearest--neighbor random walks]{Spectral analysis of 1D
nearest--neighbor random walks  and  applications to subdiffusive
trap and barrier models }

\author{A. Faggionato  }

\address{Alessandra Faggionato. Dipartimento di Matematica ``G. Castelnuovo", Universit\`a ``La
  Sapienza''. P.le Aldo Moro  2, 00185  Roma, Italy. e--mail:
  faggiona@mat.uniroma1.it}
\thanks{Work supported by the European Research Council through the ``Advanced
Grant'' PTRELSS 228032}

\begin{document}

\maketitle

\begin{abstract}  We consider
  a   sequence  $X^{(n)}$, $n \geq 1 $,   of continuous--time
nearest--neighbor random walks on the one dimensional lattice
$\bbZ$.  We reduce  the spectral analysis of the Markov generator of
$X^{(n)}$ with Dirichlet conditions outside $(0,n)$ to the analogous
problem  for  a suitable generalized second order differential
operator $-D_{m_n} D_x$, with Dirichlet conditions outside a given
interval. If  the measures $dm_n$ weakly converge to some measure
$dm_\infty$,  we prove a limit theorem for the eigenvalues and
eigenfunctions of $-D_{m_n}D_x$ to the corresponding spectral
quantities of $-D_{m_\infty}  D_x$.  As second result,  we prove the
Dirichlet--Neumann bracketing for the operators  $-D_m D_x$ and, as
a consequence, we establish
  lower and upper bounds for
the asymptotic annealed eigenvalue counting functions in the case
that $m$ is a self--similar stochastic process.  Finally, we apply
the above results to investigate the spectral structure of some
classes of  subdiffusive random trap and barrier models coming from
one--dimensional physics.
\bigskip


\smallskip

\noindent {\em MSC-class:
60K37, %processes in random environment
82C44, % dynamics of disordered systems
34B24. % sturm-liouville theory
   }

\end{abstract}

%\centerline{\bf Warning} {\sl If $m$ is c\`{a}dl\`{a}g,
%$dm(\{x\})=0$ for any $x \in \bbR$. A volte si e' confuso questo
%concetto con il fatto che $x$ is not a jump point of $m$.
%correggere!}

\section{Introduction}
Continuous--time nearest--neighbor random walks on $\bbZ$
%(also called birth--and--death processes)
 are a basic object in
probability theory with numerous applications, including the
modeling  of one--dimensional physical systems. A fundamental
example   is given by the simple symmetric random walk (SSRW) on
$\bbZ$, of which we recall some standard results. It is well known
that the SSRW converges to the standard Brownian motion under
diffusive space--time rescaling.
 Moreover,
the sign--inverted  Markov generator with Dirichlet conditions
outside $(0,n)$ has exactly $n-1$ eigenvalues, which are all
positive  and simple. Labeling  the eigenvalues in increasing order
$\bigl( \l _k^{(n)} : 1\leq k <n \bigr)$,  the $k$--th one is given
by  $ \l^{(n)}_k= 1-\cos (\p k /n)$ with  associated eigenfunction $
f^{(n)}_k (j)= \sin (k\p j /n) $, $j \in \bbZ \cap [0,n]$. Extending
$f^{(n)}_k$ to all $[0,n]$ by linear interpolation, one observes
that
$$ \lim_{n\uparrow \infty} n^2 \l_k^{(n)}= \frac{\p^2 k^2}{2}=: \l_k$$ and
$$ \lim _{n\uparrow \infty} f^{(n)}_k(   [ n x]   )= \sin (k \p x)=: f_k (x) \,,$$
where $[nx]$ denotes the integer part of $nx$ and the last limit is
in the space $C([0,1])$ endowed of the uniform norm. On the other
hand, the standard Laplacian $-(1/2)\D$ on $[0,1]$ with Dirichlet
boundary conditions has $\bigl(\l_k:k \geq 1\bigr)$ as family of
eigenvalues and $f_k$ as eigenfunction associated to the simple
eigenvalue $\l_k$.

\smallskip

Considering this simple example it is natural to  ask how general
the above considerations can be. In particular, given  a family of
continuous--time nearest--neighbor random walks $X^{(n)}$  defined
on the rescaled interval $[0,1]\cap \bbZ_n$,   $\bbZ_n:=\{ k/n\,:\,
k \in \bbZ\}$,  killed when reaching the boundary, one would like
very general criteria to establish (i) the convergence of $X^{(n)}$
to some stochastic process $X^{(\infty)}$, (ii)
 the convergence of the
eigenvalues and eigenfunctions of the Dirichlet Markov generator of
$X^{(n)}$   to the corresponding spectral quantities of the
Dirichlet Markov generator of some stochastic process
$Y^{(\infty)}$. Note that we have not imposed
$X^{(\infty)}=Y^{(\infty)}$ and the reason will be clarified soon.

 \smallskip

Criteria to establish (i) also in a more general context have been
developed by C. Stone in \cite{S}.
%, while  in the first part of this
%paper we focus on  a general criterion to
%  establish (ii).
%In order to allow a better understanding of the
%connection between the two solutions of (i) and (ii), we briefly
%recall the approach of \cite{S}. The
% starting  observation  is that  $X^{(n)}$ can be
% expressed as  $(S_n, dM_n) $--space--time change of the (suitably
% killed) standard Brownian motion $B$, for some  scale function $S_n$
% and some speed measure $dM_n$ (cf. \cite{IM}, \cite{D}, \cite{L2}).
% If $S_n$ is the identity function $\bbI$ and $dM_n$
% converges to some measure $dM$ (as for the SSRW), one can apply
% Stone's result and conclude, under suitable weak technical
% assumptions, that $X^{(n)}$ converges to the process $X^{(\infty)}$
% obtained as $(\bbI, dM)$--space--time change of $B$, suitably
% killed.
% If $S_n$ is not the identity function, by a canonical transformation  one can map $X^{(n)}$
% into a   new random walk  $Y^{(n)}$  having identity scale function and one can apply Stone's result to the
% latter. From the convergence of $Y^{(n)}$ to some process
% $Y^{(\infty)}$ one can try to recover information on a possible
% convergence of the original random walks $X^{(n)}$.
These  results  have been successfully applied in order to study
rigorously the asymptotic  behavior of nearest--neighbor random
walks on $\bbZ$ with random environment, as the random barrier model
\cite{KK},
 \cite{FJL} and  the random trap model \cite{FIN},
\cite{BC1}, \cite{BC2} (see below). In the first part of the paper,
we focus on a general criterion to establish (ii). As well known,
 by an injective map $\bbZ_n\to \bbR$,  one can always
  transform $X^{(n)}$ into a   random walk $Y^{(n)}
$  which can be expressed   as time change of the Brownian motion
$B$, suitably killed, with scale function given by the identity map
and speed measure $dm_n$. This transformation reveals crucial, since
the Markov generator of $Y^{(n)} $ can be defined
 on   continuous
and piecewise--linear functions and
  the convergence of eigenfunctions is simply in the uniform topology (otherwise one is forced to deal with
  rather complex functional spaces as in \cite{FJL}). We point out
  that in order to establish the convergence (i) one often needs to
  consider an additional transformation (thus explaining why the
  above processes
  $X^{(\infty)}$ and  $Y^{(\infty)}$ can differ).
  The sign--inverted Markov generator of $Y^{(n)}$ can
be written as a generalized differential operator $-D_{m_n} D_x$ on
a suitable interval  with Dirichlet b.c. (boundary conditions).
Briefly, in Theorem \ref{speriamo} we will show that the asymptotic
spectral structure of $-D_{m_n} D_x$  coincides with the one of
$-D_m D_x$ if the measure $dm_n$ weakly converges to the measure
$dm$, in particular we show the convergence of the $k$--th
eigenvalue and the associated eigenfunction. A  similar convergence
result is proven  by T.Uno and I. Hong in \cite{UH} for a family  of
differential operators  on $\G_n$, where  $\G_n$ is a suitable
sequence of subsets in $\bbR$ converging to the Cantor set. Some
ideas in their proof have been applied to our context, while others
are very model--dependent. The route followed here is more inspired
by modern Sturm--Liouville theory \cite{KZ}, \cite{Ze}, where the
continuity of the spectral structure is related to the continuity
properties of a suitable family of entire functions. We point out
that  continuity theorems for the spectral structure already exist.
See for example Ogura's paper \cite{O}[Section 5]. There the author
proves the vague convergence (even a stronger version)  of the so
called spectral measure $\s_n(dx) $ associated to $-D_{m_n} D_x$ to
the one $\s(dx)$ associated to $-D_m D_x$ if the measure $dm_n$
weakly converges to the measure $dm$. The spectral measure comes
from the Weyl--Kodaira--Titchmarsh theorem \cite{Ko,Y}, is a
matrix--valued measure on $\bbR$  with support coinciding with the
spectrum of the operator. One could work on Ogura's  convergence
result to deduce the convergence of  the eigenvalues and the
associated eigenfunctions. We did not follow this route since more
elaborated, preferring a more elementary approach. The same
observation holds for the continuity theorem of Kasahara
\cite{K}[Theorem 1] based on Krein's correspondence.

\smallskip

As second step in our investigation we  prove the Dirichlet--Neumann
bracketing for the generalized operator $-D_m D_x$ (Theorem
\ref{DNB}). This is a key result in order to get estimates on the
asymptotics of eigenvalues as in the  Weyl's classical theorem for
the Laplacian on bounded Euclidean domains (see \cite{W1},
\cite{W2}, \cite{CH1}, \cite{RS4}[Chapter XIII.15]). The form of the
bracketing used in our analysis  goes back to G. M\'{e}tivier and
M.L. Lapidus (cf. \cite{Me}, \cite{L}) and  has been successfully
applied in \cite{KL}  to establish an analogue of Weyl's classical
theorem for the Laplacian on finitely ramified self--similar
fractals. Finally, from the Dirichlet--Neumann bracketing we  derive
the behavior  at $\infty$ of the
 averaged
eigenvalue counting function   of the operator $-D_m D_x$ on a
finite interval with Dirichlet b.c.  under the assumption that $m$
is a self--similar stochastic process (see Theorem \ref{tacchino}).
We point out that in \cite{Fr}, \cite{H}, \cite{KL} \cite{UH} the
authors study the asymptotics of the eigenvalues for the Laplacian
defined on self--similar geometric objects. In our case, the
self--similarity structure enters into the problem through the
self--similarity of $m$.

\smallskip

As application of the above analysis (Theorem \ref{speriamo},
Theorem \ref{DNB} and Theorem  \ref{tacchino}) we  investigate  the
small eigenvalues of some classes of  subdiffusive random trap and
barrier models (Theorems \ref{affinita} and \ref{affinitabis}). Let
$\cT= \{\t_x\,:\, x \in \bbZ\}$ be a family of positive i.i.d.
random variables belonging to the domain of attraction of an
$\a$--stable law, $0<\a<1$. Given $\cT$, in the random trap model
the particle waits at site $x$  an exponential time  with mean
$\t_x$ and after that it jumps to $x-1$, $x+1$ with equal
probability. In the random barrier model, the probability rate for a
jump from $x-1$ to $x$ equals the probability rate for a jump from
$x$ to $x-1$ and is given by $1/\t(x)$. We consider also generalized
random trap models, called {\em asymmetric random trap models} in
\cite{BC1}. Let us call $X^{(n)} $ the rescaled random walk on
$\bbZ_n$ obtained by accelerating the dynamics of  a factor of order
$n^{1 +\frac{1}{\a}}$ (apart a   slowly varying function) and
rescaling the lattice by a factor $1/n$.
 As
investigated in \cite{KK}, \cite{FIN} and \cite{BC1}, the law of
$X^{(n)} $ averaged over the environment $\cT$ equals the law of a
 suitable $V$--dependent random walk $\tilde X^{(n)}$ averaged over $V$, $V$
being an $\a$--stable subordinator. To this  last random walk
$\tilde X^{(n)}$ one can apply our general results, getting at the
end some  annealed spectral information about $X^{(n)}$.
% To this last family of random
%walks one can apply the above convergence arguments based on Stone's
%result as done in \cite{KK}, \cite{FIN} and \cite{BC1}, as well as
%our general continuity theorem on the spectral structure as done
%here. The conclusion  is stated in Theorems \ref{affinita} and
%\ref{affinitabis}. Finally, we study the  distribution at $\infty$
%of the limit averaged eigenvalue counting function.

\smallskip

 Random trap and random barrier walks on $\bbZ$
  have been introduced in Physics in order  to model 1d particle or
excitation dynamics, random 1d Heisenberg ferromagnets, 1d
tight--binding fermion systems, electrical lines of conductances or
capacitances \cite{ABSO}. More recently (cf. \cite{BCKM}, \cite{BDe}
and references therein) subdiffusive random walks on $\bbZ$  have
been used as toy models for slowly relaxing systems as glasses and
spin glasses exhibiting aging, i.e. such that the time--time
correlation functions keep memory of the preparation time of the
system even asymptotically. Our results contribute to the
investigation of the spectral properties of aging stochastic models.
This analysis and the study of the relation between aging and the
spectral structure of the Markov generator has been done in
\cite{BF1} for the REM--like trap model on the complete graph.
Estimates on the first Dirichlet eigenvalue of $X^{(n)}$  in the
case of subdiffusive (also asymmetric and in $\bbZ^d$, $d\geq 1$)
trap models have been derived in \cite{Mo}, while the spectral
structure of the 1d Sinai's random walk for small eigenvalues has
been investigated in \cite{BF2}. The method developed in \cite{BF2}
is  based on perturbation  and capacity theory together with the
property that the random environment can be approximated by a
multiple--well potential. This method cannot be applied here and we
have followed a different route.
% We point out that recently in
%\cite{BCC1} and \cite{BCC2} the authors have studied the spectral
%property of Markov generators also associated to subdiffusive random
%walks by a matrix theory approach.

%\smallskip
%Finally, we mention that we have applied our spectral continuity
%theorem also to diffusive random walks improving  some previous
%results  (cf. \cite{BD}) as described in Propositions \ref{maschio1}
%and \ref{maschio2}.

\section{Model and results}\label{mod_res}
We consider a generic continuous--time nearest--neighbor random walk
$(X_t : t \geq 0 )$  on $\bbZ$. We denote by  $c(x,y)$  the
probability rate for a jump from $x$ to $y$: $c(x,y)>0$ if and only
if $|x-y|=1$, while   the Markov generator $\bbL$ of $X_t$ can be
written as \begin{equation}\label{generatore} \bbL f(x)=
c(x,x-1)\bigl[ f(x-1)-f(x) \bigr]+ c(x,x+1) \bigl[f(x+1)-f(x) \bigr]
\end{equation}
for any bounded function $f: \bbZ \rightarrow \bbR$. The random walk
$X_t$ can be described as follows: arrived at site $x \in \bbZ$, the
particle waits an exponential time of mean $1/[ c
(x,x-1)+c(x,x+1)]$, after that it jumps to $x-1$ and $x+1$ with
probability
$$ \frac{c (x,x-1)}{c(x,x-1)+c(x,x+1)}\,\;\; \text{ and } \;\;
\frac{c (x,x+1)}{c(x,x-1)+c(x,x+1)}\,,$$ respectively.

\medskip

By a recursive procedure, one can always determine two positive
functions $U$ and $H$ on $\bbZ$ such that
\begin{equation}\label{nanna}
 c(x,y)=1/\left[
 H(x) U(x \lor y ) \right] \,, \qquad \forall x,y\in \bbZ: |x-y|=1 \,.
\end{equation}
Moreover, the above functions $U$ and $H$  are univocally determined
apart a positive factor $c$ multiplying $U$ and dividing $H$.
Indeed, the system of equation \eqref{nanna} is equivalent to the
system
\begin{equation}\label{ninna}
\begin{cases}
 U(x+1)= U(x) \frac{ c(x,x-1)}{c(x,x+1) } \,,   \\
H(x)= \frac{1}{c(x,x-1) U(x) }\,, \end{cases}
 \qquad
\forall x \in \bbZ\,.
\end{equation}
We observe that
 $U $ is a constant function if and only if the jump rates $c(x,y)$
 depend  only on the starting point $x$. Taking without loss of
 generality $U\equiv 2$, we get that after arriving at site $x$ the
 random walk $X_t$ waits an exponential time of mean $H(x)$ and then
 jumps with equal probability to $x-1$ and to $x+1$. This special
 case is known in the physics literature as {\sl trap model} \cite{ABSO}.
 Similarly, we observe that $H$ is a constant function if and only
 if the jump rates $c(x,y)$ are symmetric, that is  $c(x,y)=c(y,x)$
 for all $x,y \in \bbZ$. Taking without loss of
 generality $H\equiv 1$, we get that $c(x,x-1)=c(x-1,x)= U(x)$. This special
 case is known in the physics literature both as {\sl barrier model} \cite{ABSO} and  as
 random walk among {\sl conductances}, since $X_t$ corresponds to
 the random walk associated in a natural way to  the linear resistor network with nodes
 given by the sites of $\bbZ$ and electrical filaments between
 nearest--neighbor nodes $x-1,x$ having conductance $c(x-1,x)=U(x)$ \cite{DS}.
If the rates $\{c(x,x \pm 1) \}_{x \in \bbZ}$ are random one speaks
 of random trap model, random barrier model and random
walk among random conductances.

\medskip

In order to describe some asymptotic spectral  behavior as
$n\uparrow \infty$, we consider a   family $X^{(n)}(t)$ of
continuous--time nearest--neighbor random walks on $\bbZ_n:= \{
k/n\,:\, k \in \bbZ\}$  parameterized by  $n \in \bbN_+=\{1,2,
\dots\}$. We call $c_n (x,y)$ the corresponding jump  rates and we
fix positive functions $U_n$, $H_n$ satisfying the analogous of
equation \eqref{ninna} (all is referred to $\bbZ_n$ instead of
$\bbZ$). Below we denote by $ L_n$ the pointwise operator
\begin{equation}\label{puntino} L_n f(x) = c_n (x, x-1/n )
[f(x-1/n )- f(x) ]+ c_n (x, x+1/n) [ f(x+1/n )- f(x) ]\,
\end{equation}
defined at $x\in \bbZ_n$ for all functions $f$ whose domain contains
$x-\frac{1}{n}, x, x+\frac{1}{n}$. The Markov generator of $X_t
^{(n)}$  with Dirichlet  conditions outside $(0,1)$ will be denoted
by $\bbL_n$. We recall that it is defined as the operator
$\bbL_n:\cV _n \rightarrow \cV_n$,  where
\begin{equation}\label{abbandono}
\cV _n:=\{f: [0,1]\cap \bbZ_n \rightarrow
\bbC, \; f(0)=f(1)=0\}\,,
\end{equation}
such that $$ \bbL_n f (x)=
\begin{cases}
L_n f(x) & \text{ if } x\in (0,1)\cap \bbZ_n\,,\\
0 & \text{ if } x=0,1\,.
\end{cases}
$$
As discussed in Section
 \ref{preliminare}, the  operator $-\bbL_n$  has $n-1$ eigenvalues which are
 all simple and positive, while the related eigenvectors can be taken as real vectors.
Below  we write the eigenvalues as $\l_1^{(n)}<\l_2^{(n)}< \cdots <
\l_{n-1}^{(n)}$.

\medskip
In order to determine the suitable frame  for the analysis of the
  eigenvalues and eigenvectors of $-\bbL_n$,
%As discussed in Section \ref{preliminare}, given functions $f,g$ on $[0,n]\cap \bbZ$ satisfy
% the equation $\bbL_n f(x)
%=g (x)$ for all $x \in (0,n)\cap \bbZ$ if and only if  there exist
%numbers $a,b$ such that the integral equation
%\begin{equation}\label{pioggia}
%F(x)= a+bx+\int_0^x dy \int _{[0,y) } d m_n (z) G(z) \,, \qquad
%\forall x \in \Psi_n \,,
%\end{equation}
%is satisfied, where $F\bigl(x_k ^{(n)}\bigr):=f(k)$ and
%$G\bigl(x_k^{(n)} \bigr)= g(k)$ for all $0\leq k \leq n$. We can
%identify $F$ and $G$ with the functions on $[0,
%x_n^{(n)}]=[0,S_n(n)]$ obtained by extending  respectively  $F$ and
%$G$ by linear interpolation. Using with some abuse of notation the
%same name $F$ and $G$ for these new functions, it is trivial to
%check that the integral equation \eqref{pioggia} is satisfied for
%all $x \in [0, S_n(n)]$.  We have now reached the natural frame   to
%study and state the spectral properties of the operator $\bbL_n$
%with Dirichlet conditions outside $(0,n)$.
 we
recall some basic facts from the theory of generalized second order
differential operators $-D_m D_x$ (cf. \cite{KK0}, \cite{DM},
\cite{K1}[Appendix]), initially developed to analyze the behavior of
a vibrating string. Let $m:\bbR\rightarrow [0,\infty)$ be a
nondecreasing  function with $m(x) = 0$ for all $x<0$. Without loss
of generality we can suppose that $m$ is c\`{a}dl\`{a}g.  We denote
by $dm$ the Lebesgue--Stieltjes measure associated to $m$, i.e. the
Radon measure on $\bbR$ such that $ dm( (a,b]) = m(b)-m(a)$ for all
$a<b$. We define  $E_m$ as the support of $dm$,
 i.e. the  set of
points where $m$ increases: \begin{equation}\label{emme}
 E_m :=\{ x
\in [0,\infty)\, :\, m(x-\e)<m(x+\e) \; \forall \e >0\}\,.
\end{equation}
%Moreover, for each $x \in \bbR$, we define $m_x$ as the magnitude of the jump of the function $m$ at
%the point $x$, i.e. $m_x:= m(x+)-m(x-)=m(x)-m(x-)$.
 We  suppose that  $E_m \not = \emptyset$,      $0 =  \inf
E_m$ and   $\ell_m := \sup E_m<\infty$.
%\club \textsc{Vorrei $0\leq \inf E_m$, $\ell_m\geq \sup E_m$ }
 Then, $F\in C([0, \ell_m], \bbC)$ is
an eigenfunction with eigenvalue $\l$ of the generalized
differential operator $-D_m D_x$ with Dirichlet b.c. if
$F(0)=F(\ell_m)=0$ and if it holds
\begin{equation}\label{pioggiapasso}
F(x)=  b\, x-\l \int_0^x dy \int _{[0,y) } d m (z) F (z) \,,  \qquad
\forall x \in [0, \ell_m] \,,
\end{equation}
for some constant $b$.
%We point out that  \eqref{pioggiapasso}
%together with the boundary condition $F(0)=0$ implies that $b=
%\lim_{\e\downarrow 0} \bigl( F(x+\e)-F(x)\bigr)/\e$ and that $F$
%must be linear on the intervals of $\bbR \setminus E_m$.
The number
$b$ is called {\sl derivative  number} and is denoted $F'_-(0)$ (see
Section \ref{preliminare}  for further details).
%As discussed in Section \ref{preliminare},  the
%  eigenvalues of the operator $-D_m D_x $
% with Dirichlet conditions   outside $(0,\ell_m)$ are  all simple and positive, and  form a discrete subset of $(0,\infty)$. Morevoer,
%the associated eigenfunctions can be taken as real functions.
As discussed in \cite{L1}, \cite{L2},  the operator  $-D_m D_x $
 with Dirichlet b.c. is the generator of the   quasidiffusion on $(0,\ell_m)$
 with scale function $s(x)=x$ and  speed measure $dm$, killed when reaching the boundary points $0,\ell_m$. This quasidiffusion can be suitably defined as time change of the standard one--dimensional Brownian motion \cite{L2}, \cite{S}.

\medskip

 The spectral analysis of $-\bbL_n$  can be reduced to the spectral analysis of a suitable generalized differential operator $-D_{m_n} D_x$ as follows.  We
define the function $S_n:[0,1] \cap \bbZ_n \rightarrow \bbR$ as
\begin{equation}\label{treno1}
S_n(k/n) =
\begin{cases}
0 & \text{ if } k=0 \,, \\
\sum_{j=1}^{k} U_n(j/n) & \text{ if } 1\leq k \leq n \,.
\end{cases}
\end{equation}
To simplify the notation, we set \begin{equation}\label{treno2}
x^{(n)}_k:= S_n(k/n)\,, \; \text{ for } k: 0\leq k \leq n \,.
\end{equation} Finally, we define the nondecreasing c\`{a}dl\`{a}g
function $m_n : \bbR\rightarrow [0,\infty)$ as
\begin{equation}\label{treno3}
m_n(x) =  \sum _{k=0}^n
 H_n(k/n)\bbI\bigl( x_k^{(n)} \leq  x \bigr)
\end{equation} where $\bbI(\cdot)$ denotes the characteristic function. Then
%\begin{equation}\label{speedy}
$$
dm_n = \sum _{k=0}^{n} H_n (k/n) \d_{ x _k ^{(n)} }\,, \qquad E_n:=
E_{m_n} = \{ x_k ^{(n)}\,:\, 1\leq k \leq n \}\,,\qquad  \ell_n:=
\ell_{m_n} = x^{(n)}_n\,.$$
%\end{equation}
We denote by  $C_n[0, \ell_n]$   the set of complex continuous
functions on $[0, \ell_n]$
 that are linear on $[0, \ell_n]\setminus E_n$. Then,  the map
\begin{equation}\label{gallina} T_n:\bbC ^{[0,1]\cap \bbZ_n}\ni f \rightarrow  T_n f \in
C_n[0, \ell_n]\,,
\end{equation}
associating to  $f$ the unique function $T_n f \in C_n [0, \ell_n]$
such that  $$ T_n f (x^{(n)}_k)= f (k/n)\,, \qquad 0\leq k \leq
n\,,$$ is trivially bijective. As discussed in Section
\ref{preliminare}, the map $T_n$ defines also  a bijection between
the eigenvectors of $-\bbL_n$ with eigenvalue $\l$  and the
eigenfunctions   of the differential operator $-D_{m_n} D_x$ with
Dirichlet conditions outside $(0,\ell_n)$ associated to the
eigenvalue $\l$.

\medskip

We can finally  state the asymptotic behavior of the small eigenvalues:
\begin{Th}\label{speriamo}
Suppose that $\ell_n $ converges to some $\ell \in (0,\infty)$ and
that $dm_n$ weakly converges to a measure $dm$, where $m:
\bbR\rightarrow [0,\infty)$ is a c\`{a}dl\`{a}g function such that
$m(x)=0$ for all $x \in (-\infty, 0)$. Assume that  $0=\inf E_m$,
$\ell= \sup
E_m$ and that  $dm$ is not a linear combination of a finite family of delta measures. %Moreover, suppose that $dm$ is not of
%the form $\sum _{i=1}^k a_k \d_{z_k}$, i.e. $dm$ is not a finite
%linear combination of delta measures.

Then the generalized differential operator $-D_m D_x$ with Dirichlet
conditions outside   $(0, \ell)$ has an infinite number of
eigenvalues, which are all positive and simple. List these
eigenvalues  in increasing order as
 $\{\l_k\,:\, k \geq 1\}$,
 % fix $F_k$ eigenfunction with eigenvalue $\l_k$
 and list the $n-1$ eigenvalues
of the  operator $-\bbL_n$, which are all positive and simple, as
$\l_1^{(n)} <
 \cdots< \l_{n-1}^{(n)}$.
Then for each $k \geq 1$ it holds \begin{equation}\label{europa}
\lim_{n\uparrow \infty }\l^{(n)}_k =\l_k\,.
\end{equation}
For each $k \geq 1$, fix an eigenfunction $F_k$ with eigenvalue
$\l_k$ for the  operator $-D_m D_x$ with Dirichlet conditions. Then,
by suitably choosing the
 eigenfunction $F^{(n)}_k\in C(
 [0,\ell_n])$ of  eigenvalue
$\l^{(n)}_k$ for  the operator  $-D_{m_n}D_x$ with Dirichlet
conditions, it holds
% and extending it as zero on $(\ell_n, \ell+1]$, it
%holds
\begin{equation}\label{arabia}
\lim _{n\uparrow \infty}  F^{(n)} _k =F_k\qquad \text{ in }\;   C ([0, \ell+1]) \text{\; w.r.t.\; } \|\cdot  \|_\infty\,,
 \end{equation} where $F_k$ and $F^{(n)}_k$ are  set equal to zero  on $(\ell,
 \ell+1]$ and $( \ell_n, \ell+1]$, respectively.
\end{Th}
%We recall that the weak convergence of $dm_n$ to $dm$ means that
%$\int _{\bbR} f(s) dm_n (s) \rightarrow \int_{\bbR} f(s) dm (s)$ for
%any bounded function $f \in C( \bbR)$.
 Since by hypothesis the
supports of $dm_n$ and $dm$ are all included in a common compact
subset, the above weak convergence of $dm_n$ towards $dm$  is
equivalent to the vague convergence.
The proof of the above theorem in given in Section \ref{completo}.
%while, generalizing some arguments from \cite{UH}, we describe in
%Section \ref{zzz} a constructive method to bound
 %the rate of convergence of $\l^{(n)}_k$ to $\l_k$.

  % Part of the proof follows closely the arguments  in \cite{UH}. We have isolated this part in
  % Section \ref{ondine}.
%\subsection{Subdiffusive random trap or barrier models}

\medskip

We describe now another general result relating self--similarity to
the spectrum edge, whose application will be relevant below when
studying subdiffusive random walks. Recall the definition
\eqref{emme} of $E_m$.
%It is convenient to generalize
%our previous notation. Given $\ell>0$ and given  a function $m
%:\bbR\rightarrow [0,\infty)$, c\`{a}dl\`{a}g, nondecreasing  and
%such that  $m\equiv 0$ on $(-\infty,0)$, we say that $\l$ is an
%eigenvalue of the operator $-D_m D_x$  Dirichlet conditions outside
%$(0,\ell)$ if $ \ell$ is an eigenvalue of the operator

\begin{Th}\label{tacchino}
Suppose that $m : [0, \infty) \rightarrow [0,\infty)$ is a random
process such that
 \begin{itemize}

\item[(i)] $m(0)=0$,

\item[(ii)]  $m$ is c\`{a}dl\`{a}g and increasing    a.s.,

  \item[(iii)]  $m$ has stationary increments,

  \item[(iii)]  $m$ is
self--similar, namely there exists $\a>0$ such that for all $\g>0$
the processes $\bigl( m(x)\,:\, x \geq 0\bigr)$ and $ \bigl(
\g^{1/\a} m(x/\g) \,:\, x \geq 0 \bigr)$ have the same law,

\item[(iv)]   extending $m$ to all $\bbR$ by setting $m\equiv 0$ on $(-\infty,0)$,
 for any  $x \in \bbR$ with probability one $x$ is not a jump point of $m$.
 \end{itemize}
%Define $\tilde m $ as
%\begin{equation}
%\tilde m (x) =
%\begin{cases}
%0 & \text{ if } x \leq 0 \,, \\
%m(x) & \text{ if } 0 \leq x \leq 1\,,\\
%m(\ell) & \text{ if } x \geq 1 \,.
%\end{cases}
%\end{equation}
Then, a.s.  all eigenvalues of  the  operator $- D_m D_x$ with
Dirichlet conditions outside $(0,1)$ are simple and positive, and
form a diverging sequence $\bigl(\l_k(m )\,:\, k \geq 1\bigr)$ if
labeled in increasing order. The same holds for the eigenvalues
$\bigl(\l_k(m^{-1}  )\,:\, k \geq 1\bigr)$ of  the operator $-
D_{m^{-1}}   D_x$ with Dirichlet conditions outside $(0,m(1))$,
where $m^{-1}$ denotes the c\`{a}dl\`{a}g  generalized inverse of
$m$, i.e.
\begin{equation} m^{-1} (t)= \inf\{ s\geq 0\,:\, m(s)>t\}\,,\qquad  t \geq
0\,.
\end{equation}
Moreover, if there exists $x_0>0$ such that
\begin{equation}\label{kikokoala0}
\bbE\left[ \sharp\{ k \geq 1 \,:\, \l_k (m  ) \leq x_0 \} \right] <
\infty\,,
\end{equation} then  there exist positive constants $c_1,
c_2 $ such that
\begin{equation}
 c_1  x^{\frac{\a}{1+\a} }  \leq \bbE\left[ \sharp\{ k \geq 1 \,:\,
\l_k (m  ) \leq x \} \right] \leq c_2x^{\frac{\a}{1+\a} } \,, \qquad
\forall x \geq 1\,.\label{kikokoala1} \end{equation}
 Similarly, if
there exists $x_0>0$ such that
\begin{equation}\label{kikokoala00}
  \bbE\left[ \sharp\{ k \geq 1 \,:\, \l_k (m^{-1}  ) \leq x_0 \}
\right]<\infty\,,
\end{equation}
then  there exist positive constants $c_1, c_2 $ such that
\begin{equation}
 c_1 x^{\frac{\a}{1+\a} }  \leq \bbE\left[ \sharp\{ k
\geq 1 \,:\, \l_k (m^{-1}  ) \leq x \} \right] \leq
c_2x^{\frac{\a}{1+\a} } \,, \qquad \forall x \geq
1\,.\label{kikokoala2}
\end{equation}
\end{Th}
Strictly speaking, in the above theorem  we had to write $-D_{ m_*
}D_x$ and  $-D_{(m^{-1} )_*} D_x$ instead of $- D_m  D_x$ and $-
D_{m^{-1}}   D_x$, respectively, where
\begin{equation}\label{percome} m_*(x) =
\begin{cases}
0 & \text{ if } x\leq 0\,,\\
m(x) & \text{ if } 0 \leq x \leq 1 \,,\\
m(1) & \text{ if } x \geq 1\,,
\end{cases}
\qquad (m^{-1})_*(x)=
\begin{cases}
0 & \text{ if } x\leq 0\,,\\
m^{-1} (x) & \text{ if } 0 \leq x \leq m (1)  \,,\\
m^{-1} \bigl( m(1)\bigr) & \text{ if } x \geq m(1)\,.
\end{cases}
\end{equation}
This will be understood also below, in Theorems \ref{affinita} and
\ref{affinitabis}.
 Since $m $ is c\`{a}dl\`{a}g, it has a
countable (finite or infinite) number of jumps $\{z_i\}$. For $x
\geq 0$ it holds
 \begin{equation}\label{matriosca}
 m^{-1} (x)=\begin{cases} y & \text{ if } y =m(x)\,, \; x \in [0,\infty)
 \setminus
 \{z_i\}\,,\\
z_i & \text{ if } x\in [m(z_i-), m(z_i)] \text{ for some $i$} \,.
\end{cases}
\end{equation}
Since we have assumed $E_m=[0,\infty)$ a.s.,  $m^{-1}$ must be
continuous a.s. (observe that the jumps of $m^{-1}$ correspond to
the flat regions of $m$).

\medskip
 The proof of the above theorem is given in Section
\ref{persico} and is based on the Dirichlet-- Neumann bracketing
developed in Section \ref{sec_dnb} (cf. Theorem \ref{DNB}). When $m$
is a stable subordinator \eqref{kikokoala0} and \eqref{kikokoala00}
are fulfilled  (see the proof of Theorem \ref{affinita} and
\ref{affinitabis}).

\medskip

  As
application of Theorem \ref{speriamo} and Theorem  \ref{tacchino},
% In the trap model with  $U_n\equiv 1/n$, we have  $S_n(k)=k/n$ and $dm_n= \sum
%_{k=0}^n H_n (k) \d_{k/n}$. In the barrier model with  $H_n \equiv c_n$, we
%have $dm_n= c_n \sum_{k=0}^n \d_{ x^{(n)}_K }$.
 we consider special families of
  subdiffusive random trap and barrier models (cf. \cite{ABSO},
\cite{KK}, \cite{FIN}, \cite{BC1}, \cite{BC2},  \cite{FJL} and
references therein). To this aim we fix a family  $\cT:= \{\t(x)
\,:\, x \in \bbZ\}$ of  positive i.i.d. random variables in the
domain of attraction of a one--sided $\a$--stable law, $0<\a<1$.
This is equivalent to the fact  that there exists some function
$L_1(t)$ slowly varying
 as $t \rightarrow \infty$ such that
$$ F(t)= \bbP( \t(x) > t) = L_1(t) t^{-\a}\,, \qquad t>0\,.$$
%For simplicity of notation, we restrict to the case that $\t(x)$ has
%$\a$--stable law,  such that $\bbE (e^{-\l \t(x) } )  = e^{- \l ^\a}
%$ for all $\l>0$.
Let us define the  function $h$ as
 \begin{equation}\label{roman1aia} h (t)= \inf\{ s\geq 0 \,:\, 1/F(s)\geq t\}\,.
 \end{equation}
Then, by Proposition 0.8 (v) in \cite{R} we know that
\begin{equation}\label{roman1}
h(t)=  L_2 (t)t^{1/\a}\, \qquad t>0\,, \end{equation} for some
function $L_2$ slowly varying as $t \rightarrow \infty$.
% In particular, we have
%\begin{equation}\label{roman2}  1/h(t)= t^{-1/\a}L_3 (t)\,, \qquad L_3(t)=1/L_2(t) \,,
%\end{equation}
%where the function $L_3(t)$ is itself  slowly varying as
%$t\rightarrow \infty$. %Then we set
%\begin{equation}
%c(n) = 1/h(n)= 1/ \inf \{ s\geq 0\,:\, \bbP( \t(0) >s) \leq 1/n
%\}\,.
%\end{equation}

\medskip

  Finally, we denote by
  $V$ the  double--sided $\a$--stable subordinator defined
on some probability space $(\Xi ,\cF , \cP)$ (cf. \cite{B} Section
III.2). Namely, $V$ has a.s. c\`{a}dl\`{a}g paths,  $V(0)=0$ and
$V$ has non-negative independent increments such that for all $s<t$
\begin{equation}\label{salutare}
E\Big[ \exp \big\{ -\lambda [ V(t)-V(s)] \big\} \Big] \;=\;
\exp\{-\lambda ^\a (t-s)\}\,.
\end{equation}
(Strictly speaking, inside the exponential in the r.h.s.   there
should  be an extra positive factor $c_0$ that we have fixed equal
to  $1$). The sample paths of $V$ are  strictly increasing and of
pure jump type, in the sense that
%\begin{equation}\label{canada}
$V(u) = \sum_{0<v\le u} \{ V(v) - V(v-)\}$.
%\end{equation}
Moreover, the random set $\{(u,V(u) - V(u-))\,:\, u \in \bbR, \;
V(u)> V(u-) \}$ is a Poisson point process  on $\bbR \times \bbR_+$
with intensity $c w^{-1-\alpha} du\,dw$ , for some $c>0$. Finally,
we denote by $V^{-1}$ the generalized inverse function $
 V^{-1}(t)= \inf\{ s \in \bbR\,:\, V(s) > t\}$.
 Since $V$ is strictly increasing $\cP$--a.s., $V^{-1}$ has continuous paths $\cP$--a.s.

% Moreover, $\cP$--a.s. the discontinuity points of $V$
% form an infinite countable set $\{z_i\,:\,i \in \bbN\}$ non
%containing the origin. In this case, it holds    $dV= \sum_{i \geq
%\bbN} ( V(z_i)-V(z_i-) ) \d_{z_i}$ and
%\begin{equation}
%V^{-1} (x)=
%
%\begin{cases} z & \text{ if } x=V(z)\,, \; z \in
%[0,\infty) \setminus \{z_i\}\,\\
%z_i & \text{ if } x\in [V(z_i -), V(z_i)] \text{ for some } i \in
%\bbN\,.
%\end{cases}
%\end{equation}

%Below, when referring the the operator $-D_m D_x$ on the interval
%$(0,b)$ with Dirichlet conditions, $m$ being given by $ V$ or $
%V^{-1}$, we mean the operator $-D_{\tilde m} D_x$ on $(0,b)$ with
%Dirichlet conditions,  where $\tilde m$ is given by
%\begin{equation}\label{venerdi}
%\tilde m(x)=\begin{cases} m(0)=0 & \text{ if } x\leq 0\,,\\
%m(x) & \text{ if } 0\leq x\leq b\,,\\
%m(b) & \text{ if } x\geq b\,.
%\end{cases}
%\end{equation}

For random trap models we obtain:
\begin{Th}\label{affinita} Fix $a \geq 0 $ and  let  $\cT=\{
\t(x)\}_{x \in \bbZ}$ be a family of positive i.i.d. random
variables
%s.t. $$
%\bbE( e^{-\l \t(x)}) =e^{-\l^\a} \,, \qquad \l>0\,.$$
in the domain of attraction of an $\a$--stable law, $0<\a<1$.
%Let
%the slowly varying function $L_3$ be defined as in \eqref{roman2}.
If $a>0$, assume  also that   $\t(x)$ is bounded from below by a
non--random positive constant a.s.

\smallskip

Given  a realization of $\cT$,  consider the $\cT$--dependent  trap
model $\{X(t)\}_{t\geq 0}$  on $\bbZ$ with transition rates
\begin{equation}\label{rata1} c (x,y)=
\begin{cases}   \t (x) ^{-1+a} \t (y)^{a} &  \text{ if }\;
|x-y|=1\,\\
0 &  \text{ otherwise}\,.
\end{cases}
\end{equation}
Call $ \l^{(n)}_1(\cT)< \l^{(n)}_2(\cT) < \cdots <
\l^{(n)}_{n-1}(\cT)$ the (simple and positive) eigenvalues of the
Markov  generator of $X(t) $ with Dirichlet conditions outside $(0,n
)$. Then
\begin{itemize}

\item[i)]

 For each $k \geq 1$, the $\cT$--dependent random vector
\begin{equation}\label{suono}
 \g^2 L_2 (n)  n^{1+\frac{1}{\a} }  \bigl(\l^{(n)}_1 (\cT) ,
\cdots, \l_k ^{(n)}(\cT) \bigr) \end{equation} weakly converges to
the $V$--dependent random vector
$$ \bigl( \l_1 (V) , \dots , \l_k (V) \bigr)\,,$$
where $ \g= \bbE \left( \t(0)^{-a}  \right)$, the slowly varying
function $L_2$ has been defined in \eqref{roman1} and $\{\l_k
(V)\,:\, k \geq 1\}$ denotes the family of the (simple and positive)
eigenvalues of the generalized differential operator $-D_V D_x$ with
Dirichlet conditions outside $(0,1)$.

\item[ii)]
If $a=0$ and $\bbE\left( \exp\{ -\l \t(x)\}\right )= \exp\{
-\l^\a\}$, then in \eqref{suono} the quantity $L_2(n)$ can be
replaced by the constant $1$.

\item[iii)]
 There exist positive constants
$c_1, c_2$ such that
\begin{equation}\label{leoncino}
c_1  x^{\frac{\a}{1+\a} }  \leq \bbE\left[ \sharp\{ k \geq 1 \,:\,
\l_k (V) \leq x \} \right] \leq c_2x^{\frac{\a}{1+\a} } \,, \qquad
\forall x \geq 1\,.
\end{equation}

\end{itemize}
\end{Th}
The  above random walk $X(t)$ can be described as follows: after
arriving at site $x \in \bbZ$ the particle waits an exponential time
of mean
$$
\frac{ \t( x )^{1-a} }{ \t(x-1) ^a + \t(x+1)^a}\,,
$$
after that  it jumps to $x-1$ and $x+1$  with probability given
respectively by
$$
\frac{
\t (x-1) ^a
}{ \t ( x-1) ^a+ \t (x+1) ^a} \qquad \text{ and }\qquad \frac{
\t ( x+1) ^a
}{ \t ( x-1) ^a+ \t( x+1) ^a}\,.
$$
%If we take $a=0$ we come back to the standard random trap model.
The  random walk $X(t)$ is called  random trap model following
\cite{BC1}, although according to our initial terminology the name
would be correct only when $a=0$.  Sometimes we will  also refer to
the case $a \in (0,1]$ as   {\em generalized random trap model}. The
additional assumption concerning the bound from below of $\t(x)$
when $a>0$ can be weakened. Indeed, as pointed out in the proof, we
only  need  the validity of strong LLN for a suitable triangular
arrays of  random variables.

\smallskip

Of course, one can consider also the diffusive  case. Extending the
results of \cite{BD} we get
\begin{Pro}\label{maschio1} Fix $a\geq 0$ and  let  $\cT=\{ \t(x)\}_{x \in
\bbZ}$ be a family of positive  random variables, ergodic w.r.t.
spatial translations and such that $ \bbE( \t(x)  ) < \infty$,
$\bbE( \t(x) ^{-a}) < \infty$.  Given  a realization of $\cT$,
consider the $\cT$--dependent  trap model $\{X(t)\}_{t\geq 0}$  on
$\bbZ$ with transition rates \eqref{rata1} and call $
\l^{(n)}_1(\cT)< \l^{(n)}_2(\cT) < \cdots < \l^{(n)}_{n-1}(\cT)$ the
(simple and positive) eigenvalues of the Markov  generator of $X(t)
$ with Dirichlet conditions outside $(0,n )$. Then for each $k \geq
1$ and for a.a. $\cT$,
\begin{equation}
 n^{2 }  \bbE( \t(x) ^{-a} ) \bbE( \t(x) ) \l_k ^{(n)}(\cT)\rightarrow  \p^2 k^2 \,. \end{equation}
\end{Pro}

\medskip
%In order to state our results about the random barrier models, we
%write $V^{-1}$ for the generalized inverse function on $V$:
%\begin{equation}
%V^{-1}(t)=\sup\{ s \in \bbR\,:\, V(s) \leq t\}= \inf \{ s \in
%\bbR\,:\, V(s)>t\}\,.
%\end{equation}
%It is know that $V^{-1}$ has increasing continuous paths
%$\cP$--a.s.
Let us state our results concerning  random barrier models:
\begin{Th}\label{affinitabis}
Let  $\cT=\{ \t(x)\}_{x \in \bbZ}$ be a family of positive i.i.d.
random variables
%s.t. $$
%\bbE( e^{-\l \t(x)}) =e^{-\l^\a} \,, \qquad \l>0\,.$$
  in the domain of attraction of an $\a$--stable law, $0<\a<1$.
 Given a
realization of $\cT$,   consider the $\cT$--dependent barrier
 model $\{X(t)\}_{t\geq 0} $  on $\bbZ$ with jump rates
\begin{equation}\label{festalibro} c  (x,y)=
\begin{cases}  \t (x\lor y)^{-1}
& \text{ if }
|x-y|=1\,\\
0 & \text{ otherwise}\,.
\end{cases}
\end{equation}
Call $  \l^{(n)}_1(\cT)<  \l^{(n)}_2(\cT) < \cdots <
\l^{(n)}_{n-1}(\cT)$ the eigenvalues of the Markov generator of
$X(t)$ with Dirichlet conditions outside $(0,1)$. Recall the
definition \eqref{roman1} of the positive  slowly varying function
$L_2$. Then:
\begin{itemize}

\item[i)]

For each $k \geq 1$, the $\cT$--dependent random vector
\begin{equation}\label{uomoragno}
 L_2 (n ) n^{1+\frac{1}{\a}}  \bigl(\l^{(n)}_1 (\cT) ,
\dots , \l_k ^{(n)}(\cT) \bigr)
\end{equation}
 weakly converges to the
$V$--dependent  random vector
$$ \bigl(  \l_1 (V^{-1}) , \dots ,   \l_k (V^{-1} ) \bigr)\,,$$
where
 $\{ \l_k (V^{-1})\,:\, k \geq 1\}$ denotes the
family of the (simple and positive) eigenvalues of the generalized
differential operator $-D_{V^{-1} } D_x$ with Dirichlet conditions outside
$(0,V(1))$.

\item[ii)]
If $\bbE ( e^{-\l \t(x) }) = e^{-\l^\a} $ then in \eqref{uomoragno}
the quantity $L_2(n)$ can be replaced by the constant $1$.

\item[iii)]
 There exist positive  constants $c_1, c_2$ such
that
\begin{equation}\label{leoncinobis} c_1  x^{\frac{\a}{1+\a} }  \leq \bbE\left[
\sharp\{ k \geq 1 \,:\,  \l_k (V^{-1} ) \leq x \} \right] \leq
c_2x^{\frac{\a}{1+\a} } \,,\qquad  \forall x \geq 1\,.
\end{equation}
\end{itemize}
\end{Th}

\smallskip

Again, one can consider also the diffusive  case. Extending the
results of \cite{BD} we get
\begin{Pro}\label{maschio2} Let  $\cT=\{ \t(x)\}_{x \in
\bbZ}$ be a family of positive random variables, ergodic w.r.t.
spatial translations and such that $ \bbE( \t(x)  ) < \infty$. Given
a realization of $\cT$, consider the $\cT$--dependent barrier model
$\{X(t)\}_{t\geq 0}$  on $\bbZ$ with transition rates
\eqref{festalibro} and call $ \l^{(n)}_1(\cT)< \l^{(n)}_2(\cT) <
\cdots < \l^{(n)}_{n-1}(\cT)$ the (simple and positive) eigenvalues
of the Markov  generator of $X(t) $ with Dirichlet conditions
outside $(0,n )$. Then for each $k \geq 1$ and for a.a. $\cT$,
\begin{equation}
 n^{2 }  \bbE( \t(x) ) \l_k ^{(n)}(\cT)\rightarrow \p^2 k^2 \,.
\end{equation}
\end{Pro}

\bigskip

Theorem \ref{affinita} and \ref{affinitabis} cannot be derived by a
direct application of Theorem \ref{speriamo}.  Indeed, for any
choice of the sequence $c(n)>0$,   fixed a realization of $\cT$ the
measures $dm_n $ associated to the space--time rescaled random walks
$X^{(n)}(t)=n^{-1} X\bigl( c(n) t \bigr) $  do not converge to $dV$
or $d V^{-1}$ restricted to $(0,1)$, $(0,V(1))$ respectively. On the
other hand, for each $n$ one can define a random field $\cT_n$ in
terms of the $\a$--stable process $V$, i.e. $\cT_n = F_n (V)$,
having the same law of $\cT$ \cite{BC1,FIN}. Calling $\tilde X^{(n)}
$ the analogous of $X^{(n)}$ with jump rates defined in terms of
$\cT_n$, one has that the associated measures $d\tilde m_n$ satisfy
the hypothesis of Theorem \ref{speriamo}. This explains why Theorems
\ref{affinita} and \ref{affinitabis} give an annealed and not
quenched result.
 On the other hand, for the random walks $\tilde X^{(n)}$ the result
 is quenched, i.e.
 the convergence of the eigenvalues  holds  for almost all realizations of the
 subordinator $V$.
We refer to Sections \ref{affinita_proof} and \ref{primopiano} for a
more detailed discussion of the above coupling and for the proof of
Theorems \ref{affinita} and \ref{affinitabis}.
%Finally, we point out
%that by the construction given in Section \ref{zzz}, one easily
%derives for the case (ii) in Theorem \ref{affinita}  the estimate
%that  $\bigl| \tilde \l ^{(n)}_k
% -\l_k(V)\bigr|\leq c(k) /n$ a.s., where
%$ \tilde \l ^{(n)}_k$  denotes the $k$--th Dirichlet eigenvalue of
%$\tilde X^{(n)}$.

\subsection{Outline of the paper}   The paper is structured as
follows.  In Section \ref{da_a} we  explain how the spectral
analysis of $-\bbL_n$ reduces to the spectral analysis of the
operator $-D_{m_n} D_x$. In Section \ref{preliminare} we recall some
basic facts of generalized second order operators.   In particular,
we characterize the eigenvalues of $-\bbL_n$ as zeros of a suitable
entire function.
% In Section \ref{sec_zeros} we apply some general
%theorem about the dependence on the parameter of the zeros of a
%continuously parameterized family of entire functions. In Section
%\ref{sec_minmax} we  investigate the eigenvalues of $-D_{m_n} D_x$
%using the minimum--maximum characterization. This
In Section \ref{completo} we prove  Theorem \ref{speriamo}.
 In
Section \ref{sec_dnb} we prove the Dirichlet--Neumann bracketing.
This result, interesting by itself,  allows us to prove Theorem
\ref{tacchino} in Section \ref{persico}. Finally, we move to
applications: in Section \ref{affinita_proof} we prove Theorem
\ref{affinita}, in Section \ref{primopiano} we prove Theorem
\ref{affinitabis}, while in Section \ref{diffondo} we prove
Propositions \ref{maschio1} and \ref{maschio2}.

%%%%%%%%%%%%%%%%%%%%%%%%%%%%%%%%%%%%%

\section{From
$-\bbL_n$
 to $- D_{m_n } D_x$ }\label{da_a}

 Recall the definition of the local operator $L_n$ given
in \eqref{puntino} and of the bijection $T_n$ given in
\eqref{gallina}.
\begin{Le}
Given  functions $f,g : [0,1]\cap \bbZ_n \rightarrow \bbR$, the
system of identities
\begin{equation}\label{gattino}
 L_n f (x) = g(x) \,, \qquad \forall x \in (0,1) \cap
\bbZ_n\,, \end{equation} is equivalent to the system
\begin{equation}\label{nero}
f(x)=f(0)+\sum_{j=1} ^{nx} U_n(j/n)\left( \frac{f(1/n)-f(0)}{
U_n(1/n)}+ \sum_{k=1}^{j-1} H_n(k/n) g(k/n) \right) \,, \qquad
\forall x \in (0,1]\cap \bbZ_n\,,
\end{equation}
where we convey to set $\sum_{k=1}^0 H_n(k/n) g(k/n)=0$. Setting $F=
T_n f $, $G=T_n g$ and $$ b =\frac{F(x^{(n)}_1)-F(0)}{ U_n(1/n)}-
H_n(0)G(0)\,,$$ \eqref{nero} is equivalent to
\begin{equation}\label{luce}
F(x)= F(0)+ b x+ \int_0^x dy \int _{[0,y) }G(z) dm_n(z)\,, \qquad
\forall x \in [0, \ell_n ]\,.
\end{equation}
In particular, $f: [0,1 ]\cap \bbZ_n \rightarrow \bbR$ is an
eigenvector with eigenvalue $\l$ of the operator $-\bbL_n$ if and
only if  $T_n f$ is an eigenfunction with eigenvalue $\l$ of the
generalized differential operator $-D_{m_n} D_x$ with Dirichlet
conditions outside $(0,\ell_n)$.
%In particular, $f:[0,1]\cap \bbZ_n \rightarrow \bbR$ is eigenvector
%of $-\bbL_n$ with Dirichlet conditions outside $(0,1)$, associated
%to the eigenvalue $\l$, if and only if $F=T_n f$ solves the integral
%equation \begin{equation}\label{mare}
% F(x)=\frac{F(x_1 ^{(n) } )
%}{U_n (1/n)} x -\l \int_0 ^x dy \int _{[0,y)  } F(z) d m_n (z) \,,
%\end{equation}
%together with the boundary condition $F(x^{(n)}_n)=0$.
\end{Le}
\begin{proof}  For simplicity of notation
we write $U,H$ instead of $U_n,H_n$. Moreover, we use the natural
bijection $\bbZ \ni k \rightarrow k/n \in \bbZ_n$, denoting the
point $k/n$ of $\bbZ_n$ simply as $k$.
 Setting $\D f (j)= f(j)-f(j-1)$, we can
rewrite \eqref{gattino} by means of the recursive identities
$$
\frac{\D f (j+1) }{U(j+ 1)}= H(j) g(j) + \frac{\D f(j)}{U(j)}
\,,\qquad \forall j \in (0,n)\cap \bbZ \,.$$ By a simple telescopic
argument the above system is equivalent to
%This system
%of identities  is equivalent to
%\begin{equation}\label{torlo} \D f (j+1)= U(j+1)\left( \frac{\D
%f(1)}{ U(1)}+ \sum_{k=1}^j H(k) g(k) \right)\,, \qquad \forall j \in
%(0,n) \cap \bbZ \,.\end{equation}
% Writing  $ f(x)= f(1)+ \sum _{j=2}^{x} \D f (j)$  for all $ x\in (0,n]\cap \bbZ  $,
% \eqref{torlo} becomes  equivalent to
$$ f(x)=
 % f(1)+\sum_{j=2} ^{x}= U(j)\left( \frac{\D f(1)}{ U(1)}+ \sum_{k=1}^{j-1} H(k) g(k) \right)=\\
f(0)+\sum_{j=1} ^{x} U(j)\left( \frac{\D f(1)}{ U(1)}+
\sum_{k=1}^{j-1} H(k) g(k) \right) \,, \qquad \forall x \in
(0,n]\cap \bbZ\,, $$ with the convention that the last sum is zero
if $j=1$.   This proves that \eqref{gattino} is equivalent to
\eqref{nero}. Using $T_n,F,G, m_n$ we can rewrite \eqref{nero} as
\begin{equation}\label{tempoko}
F(x)= F(0)+ \int_0^x dy \left(   \frac{F(x^{(n)}_1)-F(0)}{
U(1)}+\int _{(0,y) }G(z) dm_n(z) \right)\,, \qquad \forall x \in (0,
\ell_n ]\,.
\end{equation}
%Indeed, for $ x = x^{(n)}_{k}$, $k\geq 1$, one can write
%\begin{multline*}
%\int _0 ^x dy \int _{(0,y) }G(z) dm_n(z)=
% \sum_{j=1}^{k-1} \bigl[
%x^{(n)} _{j+1} - x ^{(n)}_j \bigr]  \sum _{i=1} ^j G\bigl( x_i ^{(n)} \bigr)dm_n (\{ x_i^{(n)} \} )=\\
%\sum_{j=1}^{k-1} U (j+1) \sum _{i=1} ^j g(i)H(i) =\sum_{j=2}^{k}
%U(j) \sum _{i=1} ^{j-1} g(i)H(i)   = \sum_{j=1}^{k} U(j) \sum _{i=1}
%^{j-1} g(i)H(i)  \,,
%\end{multline*}
%where in the last identity we have used the convention that
%$\sum_{i=1}^0 g(i) H (i)=0 $. From the above identity it is simple
%to prove that \eqref{nero} is equivalent to \eqref{tempoko} for $x=
%x^{(n)}_k$, $k\geq 1$. The validity of \eqref{tempoko} for $x \in
%\{x^{(n)}_k: k \geq 1\}$ automatically extends to all $x \in (0,
%\ell_n]$. This concludes the proof of the equivalence between
%\eqref{nero} and \eqref{tempoko}.
 Trivially,
equation \eqref{tempoko} is equivalent to \eqref{luce}. Finally, the
conclusion of the lemma follows from the previous observations and
the
 discussion about the generalized differential operator $-D_m D_x $ given in
 Section \ref{mod_res}.
\end{proof}

\section{Generalized second order differential operators}\label{preliminare}

 For the reader's convenience and for next applications,  we recall
the definition of generalized differential operator. We mainly
follow \cite{KK0}, with some slight modifications that we will point
out. We refer to \cite{KK0}, \cite{DM} and \cite{Ma} for a detailed
discussion.

Let $m:\bbR\rightarrow [0,\infty)$ be a c\`{a}dl\`{a}g nondecreasing
function with $m(x) = 0$ for all $x<0$.  We define $m_x$ as the
magnitude of the jump of the function $m$ at the point $x$:
\begin{equation}\label{jumping}
 m_x=m(x)-m(x-)\,, \qquad x \in \bbR\,.
 \end{equation} We
define $E_m$ as the support of $dm$,
 i.e. the  set of
points where $m$ increases (see \eqref{emme}).
%Moreover, for each $x \in \bbR$, we define $m_x$ as the magnitude of the jump of the function $m$ at
%the point $x$, i.e. $m_x:= m(x+)-m(x-)=m(x)-m(x-)$.
 We  suppose that  $E_m \not = \emptyset$,    $0 = \inf
E_m$ and   $\ell_m := \sup E_m<\infty$.

Given a continuous function $F(x)\in C([0,\ell_m]) $ and a
$dm$--integrable function $G$ on $[0, \ell_m]$ we write $ -D_m D_x
F=G$ if there exist complex constants $a,b$ such that
\begin{equation}\label{pioggia}
F(x)= a+bx-\int_0^x dy \int _{[0,y) } d m (z) G(z) \,,  \qquad
\forall x \in [0,\ell_m]\,.
\end{equation}
We remark
 that the integral term in equation \eqref{pioggia} can be written
also as
$$
\int_0^x dy \int _{[0,y) } d m (z) G(z)= \int _{[0,x]} (x-z) G(z)
dm(z)= \int_0^x dy \int _{[0,y] } d m (z) G(z) \,.
$$
We point out that   equation \eqref{pioggia} implies that $F$ is
linear on $[x_1,x_2]$ if $m$ is constant on $(x_1, x_2) \subset [0,
\ell_m]$.

 As discussed in \cite{KK0}, the function $G$
is not univocally determined from $F$. To get uniqueness, one can
for example fix the value of $b$ and $b-\int_{[0, \ell_m]} G(s)
dm(s)$. These values are called  {\sl derivative numbers} and
denoted by $F'_-(0)$ and $F_+' (\ell_m)$, respectively.
 Indeed,     in \cite{KK0}  the domain $\cD_m$ of the
differential operator $-D_m D_x$ is defined as the family of
complex--valued extended functions $F[x]$, given by the triple
$\bigl( F(x), \, F'_-(0),\,F'_+ (\ell_m) \bigr)$, while the authors
set  $-D_m D_x F[x]= G(x)$. We prefer to avoid the notion of
extended functions here, since not necessarily.

It is simple to check that the function $F$ satisfying
 \eqref{pioggia} fulfills the following properties: for
each $x \in [0, \ell_m)$ the function $F(x)$  has right derivative
$F'_+(x)$, for each $x \in (0, \ell_m]$ the function $F(x)$ has left
derivative $F'_-(x)$ and
 \begin{align}
 & F'_+(x)= b - \int _{[0,x]} G(s) dm (s)\,, \qquad x \in [0,
 \ell_m) \,,\label{napoleone1}\\
 & F'_- (x)= b - \int _{[0, x)} G(s) dm (s) \,, \qquad x \in (0,
 \ell_m]\,.\label{napoleone2}
 \end{align}
In view of the definition of $F'_-(0)$ and $F'_+(\ell_m) $, the
above identities extend to any $x \in [0, \ell_m]$. In addition, if
$m_0=0$ then $F'_-(0)= \lim _{ \e \downarrow 0} F'_+(\e)$, while if
$m_{\ell_m} =0$ then $F'_+(\ell_m) = \lim _{\e \downarrow 0} F'_-
(\ell_m - \e)$.

 As discussed in \cite{KK0}, fixed  $\l \in \bbC$ there exists
a unique function $F \in C([0, \ell_m])$ solving equation
\eqref{pioggia}    with $G=\l F$ for fixed $a,b$. In other words,
fixed $F(0)$ and $F'_-(0)$ there exists a unique solution of the
homogeneous differential equation
\begin{equation}\label{ferrari0}
 -D_m D_x F= \l F\,.
 \end{equation}
 Given  $\l \in \bbC$, we define $\varphi(x,\l)  $ and $\psi (x,\l)$
as the solutions \eqref{ferrari0}
  satisfying respectively the initial conditions
\begin{align}
& \varphi (0,\l) =1\,, \qquad  \varphi '_- (0,\l)=0\,,\label{ferrari1}\\
& \psi (0,\l)=0\,, \qquad  \psi'_- (0,\l)=1 \,.\label{ferrari2}
\end{align}
 It is known  that each function $F \in C([0,
 \ell_m])$ satisfying \eqref{ferrari0} is a linear combination of
 the independent solutions
$\varphi(\cdot,\l)$ and $\psi(\cdot,\l)$. Finally, $F\not \equiv 0 $
is called an eigenfunction of the operator $-D_m D_x$ with Dirichlet
[Neumann] b.c. if $F$ solves \eqref{ferrari0} for some $\l \in
\bbC$, and moreover $F(0)=F(\ell_m)=0$ [$F'_-(0)=F'_+(\ell_m)=0$].
By the above observations, we get that $F$ is a Dirichlet
eigenfunction if and only if $F$ is a nonzero multiple of $ \psi (x,
\l)$ for $\l \in \bbC$ satisfying $\psi (\ell_m, \l)=0$, while $F$
is a Neumann eigenfunction if and only if $F(x)$ is a nonzero
multiple of $\varphi  (x,\l)$ with $\l \in \bbC$ satisfying
\begin{equation}\label{storia}
\int _0 ^\ell  \varphi (s, \l) dm  (s) =0\,. \end{equation} In
particular, the  Dirichlet and the  Neumann eigenvalues are all
simple.

\medskip

We collect in the following lemma some known results concerning the
Dirichlet eigenvalues and eigenfunctions:
\begin{Le}\label{volare}
Let $m:\bbR\rightarrow  [0,\infty)$ be a nondecreasing
c\`{a}dl\`{a}g function such that $m(x)=0$ for $x<0$, $0=\inf E_m $,
$\ell_m:= \sup E_m<\infty$. Then the differential operator $-D_{m}
D_x $ with Dirichlet conditions outside $(0, \ell_m)$  has a
countable (finite or infinite) family of eigenvalues, which are all
positive and simple. The set of eigenvalues has no accumulation
points.  In particular, if there is an  infinite number of
eigenvalues $\{\l_n\}_{n\geq 1}$, listed in increasing order, it
must be $\lim _{n\uparrow \infty}\l_n=\infty$.

 The above  eigenvalues coincide with the
zeros of the entire function $\bbC\ni \l \rightarrow \psi (\ell_m,
\l) \in \bbC$.  The eigenspace associated to the eigenvalue $\l$ is
spanned by the real function $\psi (\cdot, \l)$. Moreover, $F$ is an
eigenfunction of $-D_m D_x$ with Dirichlet conditions outside
$(0,\ell_m)$ and associated eigenvalue $\l$ if and only if
\begin{equation}\label{nonnina}
F(x)= \l \int _{[0, \ell_m)} G_{0, \ell_m}(x,y) F(y) d m (y)\,,
\qquad \forall x \in [0,\ell_m]\,,
\end{equation}
where, given an interval $[a,b]$,  the  Dirichlet Green function
$G_{a,b}:[a,b ]^2 \rightarrow \bbR$ is defined as
\begin{equation}\label{barbamammaX}
G_{a,b}(x,y)= \begin{cases}
 \frac{ (y-a)(b-x) }{ b-a } & \text{ if } y \leq x\,, \\
\frac{ (x-a)(b-y)}{b-a} & \text{ if } x\leq y\,.
\end{cases}
\end{equation}
In particular, for any  Dirichlet eigenvalue $\l$ it holds
\begin{equation}\label{materazzi}
\l \geq \bigl[\ell_m m (\ell_m)]^{-1}\,.
\end{equation}
\end{Le}
\begin{proof} The first part of the lemma follows easily from the analysis
given in \cite{KK0} concerning the entire function  $\psi(\cdot,
\l_m)$. The characterization \eqref{nonnina} follows by
straightforward computations from the definition of Dirichlet
eigenfunctions. To conclude, we observe that  \eqref{nonnina}
implies  $ \|F\|_\infty \leq \l \|F\|_\infty \ell_m dm ([0,
\ell_m))$,  since trivially $0\leq G_{0, \ell_m} (x,y) \leq \ell_m$.
\eqref{materazzi} then follows. \end{proof}

 As discussed in
\cite{KK0}, page 29, the function $\varphi$ can be written as
$\l$--power series $ \varphi(s,\l)= \sum_{j=0}^\infty (-\l)^j
\varphi_j (s)$ for suitable functions $\varphi_j$. Therefore the
l.h.s. of \eqref{storia} equals $ \sum_{j=0}^\infty (-\l)^j \int
_{(0,1)} \varphi_j(s) dm (s)$. From the bounds on $\varphi_j$
%(cf.the proof of Lemma \ref{volare})
 one derives that the l.h.s. of
\eqref{storia} is an entire function in $\l$, thus implying that its
zeros (or equivalently the eigenvalues of the operator $-D_mD_x$
with Neumann b.c.) form a discrete subset of $ [0,\infty )$.
Moreover (cf. \cite{KK0}) the eigenvalues are nonnegative and $0$
itself is an eigenvalue.
\section{Proof of Theorem \ref{speriamo} }\label{completo}
We divide the proof in subsections.
\subsection{Eigenvalues as zeros of entire
functions}\label{sec_zeros}

At this point, we have  reduced  the analysis of the spectrum  of
the differential operator  $-D_m D_x$ with Dirichlet conditions
outside $(0,\ell_m)$ to the analysis of the zeros of the entire
function $\psi(\ell, \cdot)$. As in \cite{KZ} and \cite{Ze} a  key
tool is  the following result, whose proof can be found in
\cite{Di}, page 248:
\begin{Le}\label{dindon} Let $\Xi $ be  a metric
space, $f:\Xi\times  \bbC  \rightarrow \bbC $ be a continuous
function such that for each $\a \in \Xi$ the map $f(\a,\cdot )$ is
an entire function. Let $V \subset \bbC $ be an open subset  whose
closure  $\bar V$  is compact, and let $\a_0\in\Xi$ be such that no
zero of the function $f( \a_0, \cdot) $ is on the boundary of $V$.
Then there exists a neighborhood $W$ of $\a_0$ in $\Xi$ such that:
\begin{enumerate}

 \item
for any $\a \in W$, $f( \a,\cdot )$ has no zero on the boundary of
$V$,

\item
the sum of the orders of the zeros of $f(\a,\cdot )$ contained in
$V$ is independent of $\a$ as  $\a$ varies in $W$.
\end{enumerate}
\end{Le}

From now on,  let   $m_n$ and $m$ be as in Theorem \ref{speriamo}.
Given $\l \in \bbC$, define
% $\varphi(x,\l) $ and
$\psi (x,\l)$  as
the solution on the homogeneous differential equation
\eqref{ferrari0} satisfying the initial condition
% \eqref{ferrari1} and
\eqref{ferrari2}.
% respectively.
 Define similarly  %$\varphi^{(n)}(x,\l)$ and
 $\psi^{(n)}(x,\l)$ by replacing $m$ with $m_n$.
The following fact will be fundamental in the application of Lemma \ref{dindon}.

\begin{Le}\label{sanpietro}
Define  $\psi(\cdot,\l)$, $\psi(\cdot, \ell)$ as $\psi(\ell,\l)$, $\psi(\ell_n, \l)$ on $(\ell, \ell+1]$,
$(\ell_n, \ell+1]$, respectively (note that $\ell_n < \ell+1$ eventually).
Fix a sequence $\l_n \in \bbC$ converging to some $\l_\infty \in \bbC$. Then
 $\psi_n(\cdot, \l_n) $ converges to $ \psi(\cdot , \l _\infty)$ as
$n \to \infty$ uniformly in $C([0, \ell+1])$.
\end{Le}
\begin{proof} The proof is similar to the first part of the proof of Theorem 1 in \cite{K}.
 As discussed
in \cite{KK0}, page 30,     one can write explicitly the power
expansion of the entire function $\bbC \ni \l\rightarrow
\psi^{(n)}(x,\l) \in \bbC$. In particular, it holds
$
% \psi(x,\l) = \sum _{j=0}^\infty (-\l)^j \psi_j (x)\,,
 \psi^{(n)}(x,\l) = \sum _{j=0}^\infty (-\l)^j \psi^{(n)}_j
(x)$,
where $\psi^{(n)}_0(x)= x$, $\psi^{(n)}_{j+1}(x) = \int_0^x (x-s) \psi^{(n)}_j(s) dm
(s)$ for $j\geq 0$ and $x \in[0,\ell_n]$. Note that  in the above integrals  we do  not
need to  specify the border of the integration domain since the
integrand functions vanish both at $0$ and at $x$.
 By the same arguments used in \cite{KK0}[page 32] one gets  $\psi_k^{(n)}(x) \leq
 \left( \frac{x}{k+1}\right)^{k+1} \frac{m_n(x)}{k!}$ for $x\in [0, \ell_n]$.
 % and a similar bound holds for  $\psi^{(n)}(x) $ and $x \in [0, \ell_n]$.
 These bounds imply easily that
  the family $\cF$  of functions $\{ \psi^{(n)}(\cdot, \l_n) \}_n$ is uniformly bounded in $C([0, \ell+1])$.
Since
\begin{equation}\label{immacolata}
 \psi ^{(n)} (x, \l_n)= x- \l_n \int_0^x dy \int_{[0,y)} dm_n(z) \psi^{(n)} (z,\l_n)\,, \qquad x \in [0, \ell_n]\,,
\end{equation}
the above bounds imply also that the family $\cF$ is equicontinuous in   $C([0, \ell+1])$. By Ascoli--Arzel\`{a} theorem, $\cF$ is relatively compact. From the weak convergence of $dm_n$ to $dm$ and from \eqref{immacolata}, one gets that all limit functions $\tilde \psi$ satisfies
\begin{equation}\label{immacolatabis}
 \tilde \psi (x )= x- \l \int_0^x dy \int_{[0,y)} dm (z)\tilde \psi (z,\l)\,, \qquad x \in [0, \ell]\,,
\end{equation}
and is equal to $\tilde \psi(\ell)$ on $[\ell, \ell+1]$. Since the integral equation \eqref{immacolata} has a unique solution, given by $\psi(\cdot, \l)$, we get the thesis.
\end{proof}
By applying Lemma \ref{volare}, Lemma \ref{dindon} and Lemma \ref{sanpietro} we obtain:
\begin{Le}\label{nonnabruna} Let $m_n$ and $m$ be as in
Theorem \ref{speriamo}. Fix a  constant  $L>0$ different from the
Dirichlet eigenvalues of  $-D_m D_x$, and let $\{\l_i\,:\, 1 \leq i
\leq k_0\}$ be the Dirichlet eigenvalues of $-D_m D_x$ smaller than
$L$.
  Let   $\e>0$ be  such (i) $\l_{k_0}+\e<L$ and (ii)
 each
 interval $J_i:=[\l_i-\e, \l_i+\e]$ intersects $\{ \l_i\,:\, 1\leq i
 \leq k_0\}$
 only at $\l_i$, for any  $i:1\leq i \leq k_0$.
 Then
 there exists an integer $n_0$ such that:
\begin{itemize}

 \item[i)]  for all $n \geq n_0$, the spectrum of $-\bbL_n$ has only
one
 eigenvalue in $J_i$\,,

\item[ii)]
 for all $n \geq n_0$, $-\bbL_n$ has no eigenvalue inside $(0,
 L) \setminus \left( \cup _{i=1}^{k_0} J_i\right) $.
\end{itemize}
\end{Le}

\begin{proof}

We already know that the Dirichlet  eigenvalues of the operator
$-D_{m_n} D_x$ [$-D_mD_x$] are given by the zeros of the entire
function $ \psi^{(n)} (\ell_n, \cdot)$  [$\psi(\ell, \cdot )$].
Hence, it is natural to  derive the thesis by applying  Lemma
\ref{dindon} with different choices of $V$. More precisely,
% with $\Xi$, $\a_0$,
%$V$ and $f$ chosen as follows. W
we take  $\a_0 = \infty$ and $\Xi = \bbN_+ \cup \{\infty\}$
 endowed of any
metric $d$  such that all points $n \in \bbN_+$ are isolated w.r.t.
$d$ and $\lim _{n\uparrow \infty} d (n, \infty)=0$.  We  define
$f:\Xi \times  \bbC \rightarrow \bbC$ as
$$f(\a,\l)=
\begin{cases} \psi ^{(n)} (\ell_n,\l)  & \text{ if } \a=n \,,\\
\psi (\ell,\l)  & \text{ if } \a=\infty \,.
\end{cases}
$$
Finally, we choose  $V=(\l_i-\e, \l_i+\e)$ as $i$ varies in  $\{1,
\dots,  k_0\}$ and after that we take $V= (0,L ) \setminus
\left(\cup_{r=1}^{k_0} J_r \right)$. The thesis then easily follows
by applying  Lemma \ref{dindon}
 if we  prove that  $f$ is continuous. The nontrivial part is to
 prove that
%\begin{equation}\label{lunare}
$\lim_{n\uparrow \infty} \psi^{(n)} (\ell_n, \l_n) = \psi (\ell, \l )$
%\end{equation}
for any sequence  of complex numbers $\{\l_n\}_{n\geq 1} $
converging to some $\l\in \bbC$. This result follows from Lemma \ref{sanpietro}
and the equicontinuity of the family of functions  $\{\psi^{(n)}(\cdot, \l_n)\}_n$  in $C([0, \ell+1])$.
\end{proof}

\subsection{Minimum--maximum characterization of the
eigenvalues}\label{sec_minmax}

For the reader's convenience,  we list some vector spaces that will
be repeatedly used in what follows. We  introduce the vector spaces
$\cA(n)$ and $\cB (n)$  as
\begin{equation}\label{milano}\cA(n):=\left \{ f :[0,1]\cap \bbZ_n \rightarrow
\bbR\,:\, f(0)=f(1)=0\right\}\,, \qquad \cB(n)= T_n \cA (n)\,,
\end{equation}
where  the map $T_n$ has been  defined in \eqref{gallina}. Hence
$F\in \cB(n)$ if and only if (i) $F(0)=F(1)=0$, (ii) $F$ is
continuous and (iii)  $F$ is linear on all subintervals $[
x_{j-1}^{(n)}, x_j ^{(n)} ]$, $1\leq j \leq n$. Since we already
know that the eigenvalues and suitable associated eigenfunctions of
$-\bbL_n$ are real, we can think of $-\bbL_n$ as operator defined on
$\cA(n)$. Finally, given $a<b$  we write $C_0[a,b]$ for the family
of continuous functions $f:[a,b]\rightarrow \bbR$ such that
$f(a)=f(b)=0$.

\medskip
Let us  recall the min--max formula characterizing the $k$--th
eigenvalue $\l ^{(n)}_k$ of $-\bbL_n$, or equivalently of the
differential operator $-D_{m_n} D_x$ with Dirichlet conditions
outside $(0,\ell_n)$. We refer to \cite{CH1}, \cite{RS4} for more
details. First we observe the validity of the  detailed balance
equation: \begin{equation}\label{lavatrice} H_n (x) c_n\bigl(x ,
x+\frac{1}{n}\bigr)=\frac{1}{U_n (x+1/n)}= H_n (x+\frac{1}{n})
c_n\bigl(x+\frac{1}{n},x\bigr) \qquad \forall x\in \bbZ_n \,.
\end{equation}
Identifying $\cA(n)$ with $\{f: (0,1) \cap\bbZ_n \rightarrow
\bbR\}$, this implies that $-\bbL_n$ is a symmetric operator in
$L^2(
 (0,1)\cap \bbZ_n, \mu_n)$, where
 $\mu_n := \sum_{x \in (0,1)\cap \bbZ_n} H_n (x) \d_{x}$.
%  Indeed,
% writing $\tilde \mu _n = \sum _{x \in \bbZ_n} H_n(x) \d_x$, $-\tilde
% \bbL_n$ for the generator on the random walk on $\bbZ_n$ with jump
% rates $c_n(x,y)$   and  defining
%  $\tilde f: \bbZ_n \rightarrow \bbR$ as $ \tilde f (x)= f(x)\bbI \bigl(x
% \in (0,1)\bigr)$ for any $f \in \cA(n)$,  it holds
% $$ \mu_n( f, -\bbL_n g )= \tilde \mu_n (\tilde f, - \tilde \bbL_n \tilde
% g)= \tilde \mu_n (\tilde g, - \tilde \bbL_n \tilde
% f)= \mu_n (g, -\bbL_n f)\,, \qquad f,g \in \cA(n)\,.$$
% Note that the second identity follows from \eqref{lavatrice}.
  Given $f \in \cA(n)$ we write $D_n(f)$ for the Dirichlet form
 $D_n(f):=\mu_n (f, -\bbL_n f) $. By simple computations, we obtain
$$
D_n(f)=  \sum _{j=1}^n U_n (j/n)^{-1} \bigl[\, f(j/n)- f ((j-1)/n)\,
\bigr]^2\,.
$$
Note that $D_n(f)=0$  with $f \in \cA(n)$ if and only if $f\equiv
0$.
%Finally, we define $$ \Phi_n (f) =
%\frac{D_n(f)}{\mu_n(f^2)} \,, \qquad  f \in \cA(n)\,,\;  f \not
%\equiv 0\,. $$
 The min--max characterization of $\l^{(n)}_k$ is
 given by the formula
\begin{equation}\label{primavera}
 \l^{(n)}_k = \min
_{V_k}\max _{f \in V_k : f \not \equiv 0}
\frac{D_n(f)}{\mu_n(f^2)}\,,
\end{equation}
where $V_k$ varies among the $k$--dimensional subspaces of $\cA(n)$.
Moreover, the minimum is attained at $V_k=V_k ^{(n)}$, defined as
the subspace spanned by the eigenvectors $f_j^{(n)}$ associated to
the first $k$ eigenvalues  $\{\l^{(n)}_j\,:\, 1\leq j \leq k\}$.

We can rewrite the above min--max  principle  in terms of $F=T_n f$
and $dm_n$. Indeed,
 given $f \in \cA(n)$, the function  $F=T_n f $ is  linear between $x_{j-1}^{(n)}$ and
 $x_j^{(n)}$, thus implying that
 $$
   U_n (j/n)^{-1} \bigl[\, f(j/n)- f ((j-1)/n)\,
\bigr]^2=
% \\\bigl[x^{(n)}_{j}- x^{(n)}_{j-1}\bigr]^{-1} \bigl[F\bigl(x^{(n)}_{j}\bigr)- F\bigl( x^{(n)}_{j-1} \bigr) \bigr]^2 =
 \int _{x^{(n)} _{j-1} } ^{ x^{(n)}_j} D_s F (s) ^2 ds\,.
$$
Hence,
 $
D_n(f) = \int _0 ^{\ell_n} D_s F (s)^2 ds $. From this identity and
\eqref{primavera}  one easily obtains that
\begin{equation}\label{estate}
\l_k ^{(n)} = \min _{S_k } \max _{  F \in S_k\,:\, F \not = 0 } \Phi
_n (F)\,,
\end{equation}
where  $S_k$ varies among all $k$--dimensional subsets of $\cB(n)$
(recall \eqref{milano}), while for a generic function $F \in C_0[0,
\ell_n]$
 we define
  \begin{equation}\label{ferragosto}
 \Phi_n (F):= \frac{
 \int _0 ^{\ell_n} D_s F (s)^2 ds }{ \int
_0^{\ell_n}  F(s)^2 dm _n (s) }
\end{equation}
whenever the denominator is nonzero. Here and in what follows, we
write $\int _0^{\ell_n}$ instead of $\int_{[0, \ell_n]}$.

\medskip
The following observation will reveal  very useful:
\begin{Le}\label{freccia}
Let $F \in \cB(n)$ and let  $G\in C_0[0,\ell_n]$ be any function
satisfying   $F(x_j^{(n)})= G( x^{(n)} _j)$ for all $0\leq j \leq
n$. Then
\begin{equation}\label{frecciarossa}
\int _0^{\ell_n} D_s F(s)^2 ds \leq \int _0^{\ell_n} D_s G(s)^2 ds
\,.
\end{equation}
In particular, if $F \not \equiv 0 $ then $\Phi_n (F)$ and $\Phi_n
(G)$ are both well defined and
%\begin{equation}\label{lampo}
$\Phi_n (F) \leq \Phi_n (G)$.
% \end{equation}
\end{Le}
\begin{proof}
In order to get \eqref{frecciarossa} it is enough to observe that by
Schwarz' inequality it holds
\begin{multline*}
\int _{x^{(n)} _{j-1} } ^{ x^{(n)}_j} D_s F (s) ^2 ds=\frac{
\bigl[F\bigl(x^{(n)}_{j}\bigr)- F\bigl( x^{(n)}_{j-1} \bigr)
\bigr]^2
}{ x^{(n)}_{j}- x^{(n)}_{j-1} } =\\ \frac{
\bigl[G\bigl(x^{(n)}_{j}\bigr)- G\bigl( x^{(n)}_{j-1} \bigr)
\bigr]^2
}{ x^{(n)}_{j}- x^{(n)}_{j-1} }
 =
  \frac{
  \bigl[ \int _{x_{j-1} ^{(n)} } ^{ x_j ^{(n)} } D_s G(s)ds
  \bigr]^2 }{
x^{(n)}_{j}- x^{(n)}_{j-1}
  } \leq  \int _{x^{(n)} _{j-1} } ^{
x^{(n)}_j} D_s G (s) ^2 ds\,.
\end{multline*}
From \eqref{frecciarossa} one derives the last issue by observing
that $ dm_n (F^2)= dm_n (G^2)$ ($dm_n(\cdot)$ denoting the average
w.r.t. $dm_n$).
\end{proof}

We have now all the tools in order to prove that the eigenvalues
$\l^{(n)}_k$ are bounded uniformly in $n$:
\begin{Le}\label{onda0}
For each $k\geq 1$, it holds
 $\sup _{n> k } \l_k ^{(n) } =: a(k)<\infty$.
\end{Le}
\begin{proof}
Given a function $f \in C_0[0, \ell_n]$ and $n \geq 1$, we define
$K_n f$ as the unique function in $\cB(n)$ such that $ f(x_j^{(n)})=
K_n f (x^{(n)}_j) $ for all $0 \leq j \leq n$. Note that $K_n$
commutes  with linear combinations: $K_n (a_1 f_1+ \cdots + a_k
f_k)= a_1 K_n f_1 +\cdots +a_k K_n f_k$.

Due to the assumption that $dm$ is not a linear combination of a
finite number of delta measures,   for some $\e>0$ we can divide the
interval
 $[0, \ell-\e )$ in $k$
subintervals $I_j = [a_j,b _j) $ such that  $dm ( \text{int}(I_j))>0
$, $\text{int}(I_j)=(a_j,b_j)$.
%If this was false, then $dm$ should
%be of the form $dm= \sum _{i=1}^r a_i \d_{z_i}$ with $r < k$ and
%$a_i>0$ for all $i$. Then we would conclude that $\ell= \sup
%E_m=\max \{z_i: 1\leq i \leq r\}$ and $dm (\{\ell\})>0$, {\bf in
%contradiction with our assumption $dm(\{\ell\})=0$, thus concluding
%the proof of our claim.}

  Since $dm_n $ converges to $dm$ weakly, it must be $dm_n
(\text{int}(I_j) )>0$ for all $j: 1\leq j \leq k$, and for  $n$
large enough. For each $j$  we fix a piecewise--linear function
$f_j: \bbR \rightarrow \bbR$,
 with support in $I_j$ and strictly positive on $\text{int} (I_j)$.
 Since $ \ell_n \rightarrow \ell>\ell-\e$, taking $n$ large enough,
all  functions   $f_j$ are  zero outside $(0, \ell_n)$, hence we can
think of $f_j$ as function in $C_0[0, \ell_n]$. Having disjoint
supports, the functions $f_1$, $f_2$,..., $f_k$ are independent in
$C_0[0, \ell_n]$.

Trivially   $K_n f_1$, $K_n f_2$,..., $K_n f _k$ are independent
functions in $\cB(n)$ for $n$ large enough since   $dm_n
(\text{int}(I_j) )>0$ for all $j$  if $n$ is large enough. Due to
the above independence, we can apply the min--max principle
\eqref{estate}. Let us write $S_k$ for the real vector space spanned
by $K_n f_1, K_n f_2, \dots, K_n f_k$  and $\bar S_k$ for the real
vector space spanned by $f_1, f_2, \dots, f_k$. As already observed,
$S_k = K_n (\bar S_k)$.
 Using also   Lemma
\ref{freccia}, we conclude that  for $n$ large enough
%\begin{multline}\label{arance}
$$ \l ^{(n)}_k \leq \max\{ \Phi_n (f) : f\in S_k\,,\,
 dm_n (f^2)> 0    \}\leq
 \max\{ \Phi_n
( f) : f \in  \bar S_k\,,\, dm_n(f^2)>0 \}\,. $$%\end{multline}
 Take
$f =a_1 f_1 + a_2 f_2 + \cdots+ a_k f_k$ such that $ dm_n (f^2) >0$.
Since $\Phi_n(f)= \Phi_n (c f)$, without loss of generality we  can
assume that  $\sum_{i=1}^k a_i ^2=1 $. Since the functions $f_j$
have disjoint supports, it holds $(D_s f)^2= \sum _{j=1 }^k a_j^2
(D_s f_j)^2 $ a.e., while $ f^2 =\sum _{j=1 }^k a_j^2 f_j^2$. In
particular, we can write
\begin{equation}\label{mandarino} \Phi_n(f) = \frac{ \sum _{j=1} ^k a_j^2  \int
_0 ^{\ell_n} D_s f_j(s)^2 ds  }{ \sum _{j=1} ^k a_j ^2  \int _0
^{\ell_n}  f_j(s)^2 dm_n (s) }\,.
\end{equation}
Hence, for  $n $ large enough, it holds
\begin{equation}
\l^{(n)}_k \leq\frac{  \max \{\int _0 ^\ell D_s f_j(s)^2 ds : 1\leq
j \leq k\}}{ \min \{ \int _0 ^\ell f_j(s)^2 dm_n(s): 1 \leq j \leq k
\} }\,.
\end{equation}
The conclusion is now trivial.
\end{proof}

\subsection{Proof of Theorem \ref{speriamo}}\label{completobis}

Most of the work  necessary for the  convergence of the eigenvalues
has been done for proving Lemma \ref{nonnabruna} and Lemma
\ref{onda0}. Due to Lemma \ref{volare}, we know that the eigenvalues
of  $-\bbL_n$ and the eigenvalues  of the differential operator
$-D_m D_x$ with Dirichlet conditions outside   $(0,\ell) $  are
simple, positive and form a set without accumulation points. Since
$-\bbL_n$ is a symmetric operator on the $(n-1)$--dimensional space
$L^2 ((0,1) \cap \bbZ_n  , \mu_n)$, where $\mu_n$ has been
introduced in Section \ref{sec_minmax}, we conclude that $-\bbL_n$
has $n-1$ eigenvalues.

Given $k\geq 1$ we take  $ a(k) $ as in Lemma \ref{onda0} and we fix
$L\geq a(k)$ such that $L$ is not an eigenvalue of $-D_m D_x$ with
Dirichlet conditions.  Let $k_0$, $\e$ and $n_0$ be as in Lemma
\ref{nonnabruna}.
 Then for $n\geq n_0$ the following holds: in each interval
$J_i=[\l_i-\e, \l_i+\e]$ there is exactly one eigenvalue of
$-\bbL_n$ and in $[0,L)\setminus  \cup _{i=1}^{k_0} J_i$ there is no
eigenvalue of $-\bbL_n$. Since we know by Lemma \ref{onda0} that
$-\bbL_n$ has at least $k$ eigenvalues in $[0,L]$ it must be  $k
\leq  k_0$ and $\l^{(n)}_i \in J_i $ for all $i: 1\leq i \leq k$. In
particular, $ \limsup _{n\uparrow \infty} |\l^{(n)} _i - \l _i |
\leq \e $ for all $ i: 1\leq i \leq k$. Using the arbitrariness of
$\e$ and $k$ we conclude that the operator $-D_m D_x$ with Dirichlet
conditions outside $(0,\ell)$ has infinite eigenvalues satisfying
\eqref{europa}. Knowing that $\l^{(n)}_k \to \l_k$ as $n \to \infty$,
the convergence from the eigenfunction $\psi(\cdot, \l^{(n)}_k)$ to $\psi(\cdot, \l_k)$, as specified
in the theorem, follows from Lemma \ref{sanpietro}.

%%%%%%%%%%%%%%%%%%%%%%%%%%%%%%%%%%%%%%%
%%%%%%%%%%%%%%%%%%%%%%%%%%%%%%%%%%%%%%%
%%%%%%%%%%%%%%%%%%%%%%%%%%%%%%%%%%%%%%%

\section{Dirichlet--Neumann bracketing}\label{sec_dnb}

Let $m:\bbR\rightarrow [0,\infty)$ be a c\`{a}dl\`{a}g nondecreasing
function with $m(x) = 0$ for all $x<0$.  We recall that $E_m$
denotes the support of $dm$,
 i.e. the  set of
points where $m$ increases (see \eqref{emme}) and that
%Moreover, for each $x \in \bbR$, we define
 $m_x$ denotes  the magnitude of the jump of the function $m$ at
 the point $x$, i.e. $m_x:= m(x+)-m(x-)=m(x)-m(x-)$.
 We  suppose that  $E_m \not = \emptyset$,    $0 = \inf
E_m$ and   $\ell_m := \sup E_m<\infty$. We want to compare the
eigenvalue counting function for the generalized operator $-D_x D_m$
with Dirichlet boundary conditions to the same function when taking
Neumann boundary conditions. In order  to apply the
Dirichlet--Neumann bracketing as stated in Section XIII.15 of
\cite{RS4} and as developed by M\'{e}tivier and Lapidus (cf.
\cite{Me} and \cite{L}),   we need to study generalized differential
operators as self--adjoint operators on suitable Hilbert spaces.

In the rest of the section we assume that
 $ m_0=m_{\ell_m}=0$.
 The reason will become clear soon. We consider the real Hilbert
space $\cH:=L^2 ( [0,\ell_m], dm)$ and denote its scalar product as
$(\cdot, \cdot)$. When writing $\int dm(y) g(y)$ we mean $\int_{[0,
\ell_m]} dm (y) g(y)$.

\subsection{The operator $-\cL_D$} We
 define the operator $-\cL_D: \cD( -\cL_D)\subset \cH \rightarrow
\cH$ as follows. First, we set that  $f \in \cD(-\cL_D)$ if there
exists a function $g \in \cH$ such that
\begin{equation}\label{aristogatti}
f(x)= bx - \int _0^x dy \int_{[0,y)} dm (z) g(z)\,, \qquad
b:=\frac{1}{\ell_m}  \int _0^{\ell_m}  dy \int_{[0,y)} dm (z)
g(z)\,.
\end{equation}
We note that the above identity implies that $f$ has a
representative given by a continuous function in $C[0, \ell_m]$ such
that $f(0)=f(\ell_m)=0$. Moreover, by the discussion following
\eqref{pioggia} (cf. \eqref{napoleone1} and \eqref{napoleone2}) and
the assumption  $m_0=m_{\ell_m}=0$, we derive from identity
\eqref{aristogatti} that the function $g \in \cH$ satisfying
\eqref{aristogatti} is unique. Hence, we define $-\cL_D f=g$. Always
due to \eqref{napoleone1} and \eqref{napoleone2}, we know that if $f
\in \cD (-\cL_D)$, then $f$ has right derivative $D_x^+ f$ on $[0,
\ell_m)$, $f$ has  left derivative $D_x^- f$ on $(0, \ell_m]$ and
has derivative $D_x f$ on $(0, \ell_m)$ apart a countable set of
points. In particular, $f $ has derivative Lebesgue a.e. on $(0,
\ell_m)$. The operator $-\cL_D$ is simply the operator $-D_xD_m$
with Dirichlet boundary conditions thought on the space $\cH$.

\begin{Pro}\label{formica} The following holds:
\begin{itemize}

\item[(i)]

the operator $-\cL_D: \cD(-\cL_D) \subset \cH \rightarrow \cH$ is
self--adjoint;

\item[(ii)]
 consider the symmetric compact operator
$\cK: \cH \rightarrow \cH$ defined as \begin{equation}\label{coppa}
\cK g(x) = \int K(x,y) g(y) dm (y)\,, \qquad g \in \cH \,,
\end{equation} where the function $K(x,y):=G_{0,\ell_m}(x,y)$ is given by
\eqref{barbamammaX}. Then, $Ran (\cK)= \cD ( -\cL_D)$ and $ -\cL_D
\circ \cK = \bbI$ on $\cH$. In particular, the operator $-\cL_D$
  admits a complete orthonormal set of eigenfunctions  and therefore  $-\cL_D$ has pure point spectrum.
Moreover,
  the above
  eigenvalues  and eigenfunctions coincide with the ones in Lemma
  \ref{volare}.

%\item[(iii)] for all $f,\hat{f} \in \cD(-\cL_D)$ it holds
%\begin{equation}\label{formina}
%(f, -\cL_D \hat{f}) = \int _0^{\ell_m} D _x f (x) D_x \hat{f} (x)dx \,.
%\end{equation}

\end{itemize}
\end{Pro}
\begin{proof}
It is trivial to check that \eqref{aristogatti} can be rewritten as
\begin{equation}\label{camillo}
 f(x)= \int K(x,y) g(y) dm (y)\,.
 \end{equation}
Hence, by definition  $\cD (-\cL_D)= Ran ( \cK)$ and $\cL _D (\cK(g)
)= g$ for all $g \in \cH$ and $\cK$ is injective (see the discussion
on the well definition of $-\cL_D$). Since $K(x,y)=K(y,x)$, the
operator $\cK$ is symmetric. Since $K \in L^2 (dm \otimes dm ) $
($K$ is bounded and $dm$ has finite mass), by \cite{RS1}[Theorem
VI.23] $\cK$ is an Hilbert--Schmidt operator and therefore is
compact (cf. \cite{RS1}[Theorem VI.22]). In particular, $\cH$ has an
orthonormal basis $\{\psi_n\}$ such that $\cK \psi_n = \g_n \psi_n$
for suitable eigenvalues  $\g_n$ (cf. Theorems VI.16   in
\cite{RS1}). Since $\cK$ is injective, we conclude that
$\g_n\not=0$,  $\psi_n = \cK( (1/\g_n) \psi_n)  \in Ran (\cK)= \cD
(-\cL_D)$ and $ -\cL_D \psi_n= (1/\g_n) \psi_n$. It follows that
$\{\psi_n\}$ is an orthonormal basis of eigenvectors of $-\cL_D$. By
\eqref{aristogatti}, the function $\psi_n \in L^2(dm)$ must have a
representative in $C[0, \ell_m]$. Taking this representative, the
identity $\psi_n = -(1/\g_n) \cL_D \psi_n$ simply means that
$\psi_n$ is an eigenfunction with eigenvalue $1/\g_n$ of the
generalized differential operator $-D_x D_m$ with Dirichlet boundary
conditions as defined in Section \ref{preliminare}. Finally, since
$-\cL_D$ admits an orthonormal basis of eigenvectors, its spectrum
is pure point and is given by the family of eigenvalues. This
concludes the proof of point (ii).

\smallskip
In order to prove (i), we observe that $\cD (-\cL_D)$ contains the
finite linear combinations of the orthonormal basis $\{\psi _n\}$
and therefore it is a dense subspace in $\cH$. Given $f,\hat{f} \in
\cD(-\cL_D)$, let $g,\hat{g} \in \cH$ such that $f= \cK g$, $\hat{f}
= \cK \hat{g}$. Then, using the symmetry of $\cK$ and point (ii), we
obtain $ (-\cL_D f,\hat{f})= (g, \cK \hat{g})=(\cK g, \hat{g})= (f,
-\cL_D \hat{f})$.  This proves that $-\cL_D$ is symmetric. In order
to prove that it is self--adjoint we need to show that, given
$v,w\in \cH$ such that $ (-\cL_D f, v)= (f, w)$ for all $f \in
\cD(-\cL_D)$, it must be $v \in \cD(-\cL_D)$ and $ -\cL_D v=w$. To
this aim, we write $g = -\cL_D f$. Then, by the symmetry of $\cK$,
it holds $ (g,v)= (-\cL_D f,v)= (f,w)=(\cK g, w)= (g, \cK w)$.
 Since this holds for any $f \in \cD(-\cL_D)$ and therefore for any
$g \in \cH$, it must be $v= \cK w$. By point (ii), this is
equivalent to the fact that $w \in\cD (-\cL_D)$ and $w=-\cL_D v$.
This concludes the proof of (i).
\end{proof}

\subsection{The operator $-\cL_N$}  We
 define the operator $-\cL_N: \cD( -\cL_N)\subset \cH \rightarrow
\cH$ as follows. First, we say that  $f \in \cD(-\cL_N)$ if there
exist a function $g \in \cH$ and a constant $a \in \bbR$ such that
\begin{equation}\label{aristogattiN}
f(x)= a - \int _0^x dy \int_{[0,y)} dm (z) g(z) \end{equation} and
\begin{equation}\label{zeromean} \int_{[0,\ell_m)} dm (z) g(z)=0\,.
\end{equation}
We note that the above identity implies that $f$ has a
representative given by a continuous function in $C[0, \ell_m]$.
Moreover, by the discussion following \eqref{pioggia} (cf.
\eqref{napoleone1} and \eqref{napoleone2}) and the assumption
$m_0=m_{\ell_m}=0$, we derive from identity \eqref{aristogattiN}
that the function $g \in \cH$ satisfying \eqref{aristogattiN} is
unique. Hence, we define $-\cL_N f=g$. Always due to
\eqref{napoleone1} and \eqref{napoleone2}, we know that if $f \in
\cD (-\cL_N)$, then $f$ has right derivative $D_x^+ f$ on $[0,
\ell_m)$, $f$ has  left derivative $D_x^- f$ on $(0, \ell_m]$ and
has derivative $D_x f$ on $(0, \ell_m)$ apart a countable set of
points.  In addition,   $D_x^+ f (0)$ and $ D_x^- f(\ell_m)$ are
zero due to \eqref{aristogattiN} and \eqref{zeromean}.  The operator
$-\cL_D$ is simply the operator $-D_xD_m$ with Neumann boundary
conditions thought of on the space $\cH$.

\begin{Pro}\label{formicaN} The following holds:
\begin{itemize}

\item[(i)]

the operator $-\cL_N: \cD(-\cL_N) \subset \cH \rightarrow \cH$ is
self--adjoint;

\item[(ii)]  the operator $-\cL_N$
  admits a complete orthonormal set of eigenfunctions and therefore  $-\cL_N$ has only pure point
  spectrum. The eigenvalues and eigenfunctions are the same as the
  ones associated to the operator $-D_x D_m$ with Neumann boundary
  conditions as defined in Section \ref{preliminare}.
\end{itemize}
  \end{Pro}
  \begin{proof}
We start with point (i). First we prove that $-\cL_N$ is symmetric.
Take $f,g,a$ as in \eqref{aristogattiN} and \eqref{zeromean}, and
take $\hat f, \hat g, \hat a $ similarly. Then,
$$ (f, -\cL_N \hat{f}) = \int dm (x) f(x) \hat{g} (x)= a \int dm(dx)
\hat{g} (x) - \int dm(x) \hat{g} (x) \int _0^x dy \int_{[0,y)} dm(z)
g(z)\,.$$ Using that $\int dm(x) \hat{g} (x)=0$ by \eqref{zeromean},
we conclude that
 $$(f,- \cL_N  \hat{f})= \int dm (x) \int dm (z) \hat{g}(x)  g( z)
\bbI_{z\leq x} (z-x)\,.
$$
Since, by \eqref{zeromean} and its analogous version for $\hat{g}$,
it holds $\int dm (x) \int dm (z) g(x) \hat g( z) (z-x)=0$, we can
rewrite the above expression in the  symmetric form
\begin{equation}\label{rovinoso}
(f,- \cL_N \hat{f})=- \frac{1}{2} \int dm (x) \int dm (z) \hat{g}(x)
g( z) |x-z|\,,
\end{equation}
which immediately implies that $-\cL_N$ is symmetric.

\smallskip

Let us consider the Hilbert subspace $\cW = \{f \in \cH\,:\,
(1,f)=0\}$, namely $\cW$ is the family of functions in $  \cH$
having zero mean w.r.t. $dm$. Then we define the operator $ T: \cH
\to \cH$ as \begin{equation}\label{mou} Tg(x) = -\int _0^x dy \int
_{[0,y)} dm(z) g(z)= \int dm(z) g(z) (z-x) \bbI_{0\leq z \leq x}\,.
\end{equation}
Finally, we write $P: \cH \to \cW$ for the orthogonal projection of
$\cH$ onto $\cW$: $Pf= f-(1,f)/(1,1)$.  Note that $ [P\circ T]g (x)
% = C(g)-\int _0^x dy\int _{[0,y)} dm(z) g(z)
= \int d m(z) g(z) H(x,z)$,
where
%\begin{align*}
%& C(g)= \int m(dz) g(z) \int _{(z, \ell_m)} m(du) (u-x)\,, \\
$$ H(x,z)= (z-x) \bbI _{0\leq z\leq
x }- \int _{(z, \ell_m)} dm(u) (z-u)\Big/ \int dm (u)   $$
%\end{align*}
Since  $H \in L^2 (dm \otimes dm )$,  due to \cite{RS1}[Theorem
VI.23] $P\circ T$ is an Hilbert--Schmidt operator on $\cH$, and
therefore a compact operator. In particular, the operator $W: \cW\to
\cW$ defined as the restriction of $P\circ T$ to $W$ is again a
compact operator. We claim that $W$ is symmetric. Indeed, setting
$f= W g$ and $f'= W  g'$, due to the first identity in \eqref{mou}
we get that $f,f' \in \cD(-\cL_N)$ and $ -\cL_N f =g$, $-\cL_N
f'=g'$. Then, using that $\cL_N$ is symmetric as proven above, we
conclude
$$ (W g, g')=(f, -\cL_N f')= ( -\cL_N f, f')= (g, W g') \,.
$$
Having proved that $W$ is a symmetric compact operator, from
\cite{RS1}[Theorem VI.16] we derive that $\cW$ has an orthonormal
basis $\{\psi_n\}_n$ of eigenvectors of $W$, i.e. $W \psi_n = \g_n
\psi_n$ for suitable numbers $\g_n$. Since  $W$ is injective (recall
the discussion on the well definition of $-\cL_N$), it must be $\g_n
\not =0$. From the identity $ W \psi_n = \g_n \psi _n$ we conclude
that
$$ \psi _n (x) = a_n - \frac{1}{\g_n}\int _0^x dy \int _{[0,y)} dm(z)
\psi_n(z)$$ for some constant $a_n \in \bbR$. The above identity
implies that $ \psi_n \in \cD(-\cL_N)$ and $-\cL_N \psi_n = (1/\g_n)
\psi_n$. On the other hand $1\in \cD(-\cL_N)$ and $-\cL_N 1 =0$.
Since $ \cH = \{ c : c \in \bbR\} \oplus \cW$, we obtain that $\cH$
admits an orthonormal basis of eigenvectors of  $-\cL_N$. This also
implies that $-\cL_N$ has only pure point spectrum. Trivially, all
eigenvectors (as all elements in $\cD(-\cL_N)$) are continuous and
are eigenvectors of $-D_x D_m$ with Neumann b.c. in the sense of
Section \ref{preliminare}. This concludes the proof of (i) and
(ii).\qedhere

\end{proof}

\subsection{The quadratic forms $q_D$ and $q_N$}

Consider now the symmetric form $q_N$ on $\cH$ with domain $Q(q_N)$
given by the elements $f \in \cH$ having a representative $f$ which
satisfies
\begin{itemize}
 \item[(A1)] $f$  is absolutely
continuous on $[0, \ell_m]$,

\item[(A2)] $\int _0^{\ell_m} D _x f (x) ^2 dx< \infty$,

\item[(A3)]
  $ D_x f$  is constant on each connected component of $(0,\ell_m) \setminus
  \text{supp}(dm)$, $\text{supp}(dm)$ being the support of the measure
  $dm$.
\end{itemize}
and  such that $q_N(f,\hat f)= \int _0^{\ell_m} D _x f (x) D_x
\hat{f} (x) dx $ for all $f, \hat f \in Q(q_N)$. In addition, we set
$q_N(f):=q_N(f,f)$.  We point out that one cannot apply directly the
theory discussed in \cite{FOT}[Example 1.2.2] since the fundamental
condition (1.1.7) there can be violated in our setting. Some care is
necessary. First of all we need to prove that $q_N$ is well defined:
\begin{Le}\label{yomo} The representative $f$ satisfying the above
properties (A1),(A2),(A3) is unique.
In  particular the form $q_N$ is well defined.
\end{Le}
\begin{proof}
Take two functions  $f,\hat f$ on $[0,\ell_m]$ satisfying the above
properties (A1),(A2),(A3) and such that $f= \hat f $ $dm$--a.e.
 We denote by $\cC$ the support of
$dm$. We first show that $f= \hat f$ on $\cC$. Suppose that $x \in
\cC$. Then for each $\e>0$ the set $I_\e:=(x-\e,x+\e)\cap [0,
\ell_m]$ has positive $dm$--measure and therefore there exists
$x_\e\in I_\e$ such that $f(x_\e)=\hat f (x_\e)$ (otherwise $f$ and
$\hat f$ would differ on a set having positive $dm$--measure). By
taking the limit $ \e \downarrow 0$ and  by continuity (property
(A1))
 we conclude that $f(x)= \hat
f(x)$ as claimed. Consider now the  open set $[0,\ell_m] \setminus
\cC$ and take  one of its connected components $(a,b)$ (recall that
$0,\ell_m\in \cC$). By property (A3) it must be $f(x)-\hat
f(x)=c_0x+c_1$ on $(a,b)$ for a suitable constants $c_0,c_1$. Since
$a,b \in \cC$ and $f=\hat f $ on $\cC$ we get that $c_0=c_1=0$, thus
proving that $f=\hat f$ on $(a,b)$. This allows to conclude.
\end{proof}
Below, when handling  with $f \in Q(q_N)$ it will be understood that
we refer to the  representative satisfying the above properties
(A1),(A2),(A3).

\begin{Le}\label{mitrino}
The form  $q_N$ is closed. Equivalently, the space $Q(q_N)$ endowed
of the scalar product
$$ (f,g)_1=  q_N (f , g )+ (f,g)\,, \qquad f ,g\in
Q(q_N)
$$
is an Hilbert space. \end{Le} \begin{proof} Take a $\|\cdot
\|_1$--Cauchy sequence $(f_n)_{n\geq 0}$ in $Q(q_N)$.
%We consider the function $g_n(x):= f_n(x) -f_n(0)$ in
% $C([0, \ell_m])$. Since $|f_n(x)-f_n(y)|^2\leq
% q_N(f_n) |x-y|\leq C |x-y| $  for a constant $C$ independent from
% $n$, we get that the family $g_n$ is equicontinuous and uniformly bounded.
% By Ascoli--Arzel\`{a} theorem, there exists a subsequence $g_{n_k}$
% converging to some $g \in C([0,\ell_m])$  uniformly. Then $g_{n_k}$
% converges to $g$ in $\cH$.
Since $D_x f_n$ is Cauchy in $L^2(dx)$, it converges to some
function $u \in L^2(dx)$. Therefore, due top Schwarz inequality,
\begin{equation}\label{confusa}  f_{n}(x) -f_{n}(0) =\int _0^x D_x f_n (z) dz \to
\int_0^x u(z) dz:= g(x) \end{equation} uniformly in $x \in
[0,\ell_m]$. Since $(f_n)_{n\geq 0}$ is a Cauchy  sequence in $\cH$,
we have that $f_n$ converges to some $f$ in
 $\cH$. Having $f_{n}-f_{n}(0)\to g$ uniformly and therefore in $\cH$, and $f_n \to f$ in
$\cH$, it must be $f_{n}(0) \to f-g$ in $\cH$. In particular, the
sequence of numbers $f_{n}(0) $ converges  to $ \int (f-g) dm/ \int
dm$. This result implies that $f_n$ converges uniformly to the
absolutely  continuous function $h:=g + \int (f-g) dm/ \int dm $ on
$[0,\ell_m]$ such that $D_x h=u$. In particular, $h$ must be linear
on the connected components of $[0,\ell_m] \setminus
\text{supp}(dm)$. Hence $h \in Q(q_N)$ and, due
 to the previous considerations,  $f_n $ converges to $h$ w.r.t. the
 norm $\|\cdot\|_1$.
\end{proof}

Finally, we define another symmetric form $q_D$ on $\cH$ with domain
\begin{equation}\label{codabiz}
Q(q_D):=\{ f \in Q(q_N)\,:\, f(0)=f(\ell_m)=0 \}
 \end{equation}
setting  $q_D(f,\hat f):=q_N(f,\hat f)= \int _0^{\ell_m} D _x f (x)
D_x \hat{f}
(x) dx $. %Note that $Q(q_D)$ has codimension 2 in $Q(q_N)$.

To each closed symmetric form on $\cH$ one associates in a canonical
way a nonnegative definite self--adjoint operator on $\cH$ (see
\cite{FOT}[Theorem 1.3.1],\cite{RS1}[Chapter VIII].
\begin{Le}\label{santostefano} The following holds:
\begin{itemize}

\item[(i)]
 The forms $q_N, q_D$ are the canonical closed symmetric forms
associated to $-\cL_N,-\cL_D$, respectively.

\item[(ii)]
$Q(q_D)$ is a closed subspace of the Hilbert space $\bigl( Q(q_N),
(\cdot, \cdot)_1\bigr)$ with  codimension $2$.

\item[(iii)] The inclusion map
$$ \iota: \bigl( Q(q_N), \|\cdot\|_1\bigr)\ni f \to f \in \bigl(
\cH, \|\cdot \|\bigr)$$ is a continuous compact operator.

\end{itemize}
\end{Le}
\begin{proof}
{\sl Item (i)}. We first focus on $q_N,-\cL_N$. We already know that
$q_N$ is a closed symmetric form. Trivially, $\cD(-\cL_N)$ is
included in $Q(q_N)$. We claim that
\begin{equation}\label{festina}
(-\cL _N f, v)= \int_0^{\ell_m} D_x f (x) D_x v (x) dx \,, \qquad
\forall f \in \cD(-\cL_N)\,, \; v \in Q(q_N)\,.
\end{equation}
By Proposition \ref{formicaN} the operator $-\cL_N$ is
self--adjoint, while by the above claim it is also symmetric and
nonnegative definite. Moreover, our claim \eqref{festina} together
with \cite{FOT}[Corollary 1.3.1] implies that $q_N$ is canonically
associated to $-\cL_N$.

To prove \eqref{festina} assume \eqref{aristogattiN} and
\eqref{zeromean} with $g \in \cH$. Then $D_x f(x)=-\int_{[0,x)}
dm(z) g(z)$ and
% and
%$$ (-\cL_N f,v)= \int _0^{\ell_m} dm(y) v(y) \int_{[0,y)} dm(z) g(z)
%= \int _{[0, \ell_m)} dm(z) g(z) \int
% $$
\begin{multline*}
 \int_0^{\ell_m} D_x f
(x) D_x v(x) dx  = - \int _0^{\ell_m} dx D_x v(x) \int_{[0,x)} dm(z)
g(z) \\=- \int_{[0,\ell_m)} dm(z) g(z) \int _{(z,\ell_m]} dx D_x
v(x) = \int_{[0,\ell_m)} dm(z) g(z)\bigl( v(z)- v (\ell_m)\bigr)\\=
(g,v)= (-\cL_N f,v)\,.
\end{multline*}
Note that in the forth identity we used \eqref{zeromean}.

\smallskip
Let us now prove the correspondence between $q_D$ and $-\cL_D$.
First we  show that $q_D$ is closed. To this aim, take $f_n \in
Q(q_D)$ such that $f_n$ is $\|\cdot\|_1$--Cauchy. Since $q_N$ is
closed, we know that there exists $f\in Q(q_N)$ with $\|f-f_n\|_1
\to 0$ as $n \to\infty$. Reasoning as in Lemma \ref{mitrino}, we
deduce that  $f_n$ converges to $f$ in the uniform norm, thus
implying that $f(0)=f(\ell_m)=0$. This proves the closeness of
$q_D$.

Knowing that $q_D$ is a closed symmetric form and reasoning as for
$q_N,-\cL_N$, to conclude
 we
only need to show that
\begin{equation}\label{festinabis}
(-\cL _D f, v)= \int_0^{\ell_m} D_x f (x) D_x v (x) dx \,, \qquad
\forall f \in \cD(-\cL_D)\,, \; v \in Q(q_D)\,.
\end{equation}
To this aim we assume \eqref{aristogatti} for some $g \in \cH$. Then
\begin{multline*}
 \int_0^{\ell_m} D_x f
(x) D_x v(x) dx  =  -\int _0^{\ell_m} dx D_x v(x) \int_{[0,x)} dm(z)
g(z) \\=- \int_{[0,\ell_m)} dm(z) g(z) \int _{(z,\ell_m]} dx D_x
v(x) = (g,v)= (-\cL_D f,v)\,.
\end{multline*}
Note that in the first identity and in the third one we used that
$v(0)=v(\ell_m)=0$.

\smallskip
{\sl Item (ii).} The thesis  follows from item (i), the definition
of $Q(q_N)$ and $Q(q_D)$.

\smallskip

{\sl Item (iii).} Since $\|f\|\leq \|f\|_1$ for each $f \in Q(q_N)$,
the inclusion map $\iota$ is trivially continuous. In order to prove
compactness, we need to show that each sequence $f_n \in Q(q_N)$
with $\|f_n\|_1 \leq 1$ admits a subsequence $f_{n_k}$ which
converges in $\cH$. Since $\|f_n\|_1 \leq1$ it holds  $ |f_n(x)-
f_n(y)| \leq \sqrt{y-x}$ for all $x,y \in [0, \ell_m]$. Applying
Ascoli--Arzel\`{a} Theorem, we then conclude that $f_n$ admits a
subsequence $f_{n_k}$ which converges in  the space $C([0, \ell_m])$
endowed of  the uniform norm. Trivially, this implies the
convergence in $\cH$.
\end{proof}

As a consequence of the above result we get that
\begin{equation}\label{sanvalentino}
  0 \leq -\cL_N\leq -\cL_D
  \end{equation}
   according to the definition on
\cite{RS4}[page 269].  For the reader's convenience and  for later
use, we recall the definition given in \cite{RS4}[page 269]:  given
nonnegative self--adjoint operators $A,B$, where $A$ is defined on a
dense subset of a Hilbert space $\cH'$ and $B$ is defined on a dense
subset of a Hilbert subspace $\cH'_1 \subset \cH'$, one says that $0
\leq A \leq B$ if (i) $Q(q_A) \supset Q(q_B)$,  and (ii) $0 \leq
q_A(\psi) \leq q_B(\psi)$ for all $ \psi \in Q(q_B)$, where $Q(q_A)$
and $Q(q_B)$ denote the domains of the quadratic forms $q_A$ and
$q_B$ associated to the operators $A$ and $B$, respectively.

\medskip

Considering the space $Q(q_N)$ endowed of the scalar product
$(\cdot, \cdot)_1$, the above Lemma \ref{santostefano} implies that
$\bigl( Q(q_N), \cH, q_N (\cdot, \cdot) \bigr)$ is a variational
triple (cf. \cite{Me}[Section II-2]). Indeed, the following holds:
(i) $Q(q_N)$ and $\cH$ are Hilbert spaces, (ii) the inclusion map
gives a continuous injection of $Q(q_N)$ into $\cH$, (iii) $q_N
(\cdot, \cdot)$ is a continuous scalar product on $Q(q_N)$ since
$|q_N(f,g)| \leq \|f\|_1\|g\|_1$ for all $f,g \in Q(q_N)$, (iv) the
scalar product $q_N(\cdot, \cdot)$ is coercive with respect to
$\cH$: $ \|f\|_1^2 - \|f\|^2 \leq q_N(f,f)$ for all  $ f \in
Q(q_N)$.

We denote by $\cN^{[0,\ell_m]} _{D,m}(x)$ the number of eigenvalues
of $-\cL_D$ not larger than $x$. Similarly we define
$\cN^{[0,\ell_m]} _{N,m}(x)$.
 By Lemma \ref{santostefano} the inclusion map $\iota:
Q(q_N)\hookrightarrow \cH$ is compact and $Q(q_D)$ is a closed
subspace in $Q(q_N)$. Applying Proposition 2.9 in \cite{Me} we get
the equality $\cN_{m,N}^{[0, \ell_m]}(x)= N(x; Q(q_N), \cH, q_N) $
and $\cN_{m,N}^{[0, \ell_m]}(x)= N(x; Q(q_D), \cH, q_D) $, where the
functions $N(x; Q(q_N), \cH, q_N) $ and $ N(x; Q(q_D), \cH, q_D)$
are defined in \cite{Me}[Page 131].
  As byproduct of Lemma \ref{santostefano}, Proposition 2.7 in \cite{Me} and the arguments used in Corollary
4.7 in \cite{KL}, we obtain that
\begin{equation}\label{agognatameta}
 \cN^{[0,\ell_m]} _{D,m}(x) \leq
\cN^{[0,\ell_m]} _{N,m}(x) \leq \cN^{[0,\ell_m]} _{D,m}(x) +2\,,
\qquad \forall x \geq 0\,.
\end{equation}
We point out  that the first inequality follows also from
\eqref{sanvalentino} and the lemma preceding Proposition 4 in
\cite{RS4}[Section XIII.15].

\medskip

\
%\subsection{ The operators $-\cL_D^I$, $-\cL_D^N$}

Up  to now we have defined $-\cL_D$ and $-\cL_N$ referring to the
interval $(0, \ell_m )$, where $0 = \inf E_m$, $\ell_m = \sup E_m$,
$m_0=0$ and $m_{\ell_m}=0$. In general, given an open interval
$I=(u,v) \subset (0, \ell_m)$, such that
\begin{equation}\label{ospedale} m_u=m_v=0, \qquad
dm\bigl((u,u+\e)\bigr)>0\text{ and } dm\bigl((v-\e,v) \bigr)>0 \;
\forall \e>0\,, \end{equation}
 we define $-\cL_D^I, -\cL_N ^I$ as the operators $-\cL_D$
and $-\cL_N$ but with the measure $dm$ replaced by its restriction
to $I$. For simplicity, we write $L^2(I,dm)$ for the space $L^2(I,
\wt{dm})$ where $\wt{dm}$ denotes the restriction of $dm$ to the
interval $I$. Then, $f \in \cD(-\cL^I_D)\subset L^2( I ,dm)$ if and
only if there exists $g\in L^2 (I,dm)$ such that, writing $I=(u,v)$,
$$ f(x)=b(x-u) - \int _u^x dy \int _{[u,y)} dm(z) g(z) \,, \qquad
\forall x \in I\,, $$ where
 $b=(v-u)^{-1}  \int _u^v dy \int _{[u,y)} dm(z) g(z) $.
The above $g\in L^2(I, dm)$ is unique and one sets $-\cL_D^I f=g$.
The definition is similar for $-\cL_N^I$. Propositions \ref{formica}
and \ref{formicaN} extend trivially to $-\cL_D^I$ and $-\cL_N^I$.
%Indeed, the restriction of $dm$ to $I$ equals $d\bar m$,
%where the function $\bar m$ is defined as $\bar m(x)= m(u)
%\bbI(x\leq u) +m(x)\bbI(x\in I)+m(v)\bbI(x\geq v)$.
  We write
$q_D^I, q_N^I$ for the corresponding quadratic forms. Finally, for
$x \geq 0$ we define
\begin{align}
& \cN ^I_{m,D} (x):= \sharp \{ \l \in \bbR:\l \leq x ,\;  \l \text{
is eigenvalue of } -\cL^I_D \}\,,\label{annibale1}
\\
& \cN^I_{m,N}(x) := \sharp \{ \l \in \bbR:\l \leq x ,\;  \l \text{
is eigenvalue of } -\cL^I_N \}\,.\label{annibale2}
\end{align}
%Note that \eqref{passetto} can be rewritten as $\cN^{(0, \ell_m)
%}_{m,D} (x) \leq \cN^{(0, \ell_m) }_{m,N} (x)$.

\begin{Le}\label{asdrubale}
Let $I_1=(a_1, b_1)$,...,$I_k=(a_k, b_k)$ be  a finite family of
disjoint open intervals, where $a_1<b_1\leq a_2<b_2 \leq a_3
<\cdots\leq a_k< b_k$ and
\begin{align*}
& m_{a_r}=0\,,\; \; m_{b_r}=0\;\; \forall r=1, \dots, k\,,\\
& dm\bigl((a_r, a_r+\e)\bigr)>0\,,\;\; dm\bigl( (b_r-\e, b_r)
\bigr)>0 \;\; \forall \e>0, \forall r=1, \dots k\,.
\end{align*}
Then
 for any $x \geq 0$  it holds
$ \cN ^{(a_1, b_k)} _{m,D} (x) \geq  \sum_{r=1}^k \cN ^{(a_r,
 b_r)}_{m,D}
(x)$.
% \label{farfalla1} \end{equation}
%& \cN ^{(a_1, b_k)} _N (x)
%\leq \sum_{r=1}^k \cN ^{(a_r, b_r)}_N (x) \,. \label{farfalla2}
%\end{align}
%\begin{split}\sharp \{ \l \in \bbR:\l \leq x ,\;  \l \text{ is } &\text{
%eigenvalue of } -\cL_D \}\geq \\ & \sum_{r=1}^k \sharp \{ \l \in
%\bbR:\l \leq x ,\;  \l \text{ is } \text{ eigenvalue of }
%-\cL_D^{I_r} \}\,.
%\end{split}
%\end{equation}
If in addition  the intervals $I_r$ are neighboring, i.e.   $b_r=
a_{r+1}$ for all $r=1,\dots, k-1$, then for any $x \geq 0$ it holds
%\begin{equation}\label{farfalla2}
$\cN ^{(a_1, b_k)} _{m,N} (x) \leq \sum_{r=1}^k \cN ^{(a_r,
b_r)}_{m,N} (x)$.
%\end{equation}
\end{Le}
The above result is the analogous to Point c) in Proposition 4 in
\cite{RS4}[Section XIII.15].
\begin{proof}
We begin with the superadditivity (w.r.t. unions of intervals) of
$\cN ^{(\cdot )} _{m,D} (x)$. We consider the direct sum $ \oplus
_{r=1}^k L^2 (I_r, dm)$. We define $A= \oplus _{r=1}^k (-\cL_D^{I_r}
)$ as the operator with
 domain $$ \cD(A)= \oplus_{r=1}^k \cD\bigl(
 -\cL_D ^{I_r}\bigr) \subset  \oplus _{r=1}^k L^2 (I_r, dm) $$ such that
 $ A \bigl[( f_r )_{r=1}^k\bigr]=  \bigl( -\cL_D ^{I_r}
 f_r\bigr)_{r=1}^k$. Due to the properties listed in \cite{RS4}[page 268] and due to
Proposition \ref{formica}, the operator
 $A$ is a nonnegative self--adjoint operator.

\smallskip

 Trivially, the map
 $
 \psi : \oplus _{r=1}^k L^2 (I_r, dm)\rightarrow L^2 ([a_1,b_k]
,dm )$
 where
$$
 \psi \bigl[ (f_r)_{r=1}^k\bigr] (x)= \begin{cases} f_r (x) & \text{
if } x \in I_r \text{ for some } r\,,\\
0 & \text{ otherwise}\,,
\end{cases}
$$
is  injective and conserves the norm. In particular, the image of
$\psi$ is a closed (and therefore Hilbert) subspace of $L^2([a_1,
b_k ],dm)$. Consider, the operator
$$A': \psi( \cD(A) ) \subset \psi \left[\oplus _{r=1}^k L^2 (I_r, dm
)\right] \rightarrow \psi\left[ \oplus _{r=1}^k L^2 (I_r, dm
)\right]\,,$$ defined as $A' ( \psi (f))= \psi ( A f)$ for all $f
\in \cD (A)$. Then, $A'$ is a nonnegative self--adjoint operator.
Due to property (3) on page 268  of \cite{RS4}  and the
characterization of the  form domain  $Q(q_D)$, we get that $-\cL_D
^{(a_1, b_k)} \leq A'$.  The   superadditivity then follows from the
lemma stated in \cite{RS4}[page 270] and property (5) on page 268 of
\cite{RS4}.
\medskip

In order to prove subadditivity of $\cN ^{(\cdot)} _{m,N} (x)$ under
the hypothesis $b_r= a_{r+1}$ for all $r=1,\dots, k-1$, we first
observe that the above map $\psi$   is indeed an isomorphism of
Hilbert spaces (recall that $m_{a_r}=0$ and $m_{b_r}=0$). From the
definition of $q_N$ and Lemma \ref{santostefano} it is trivial to
check that
$$ 0 \leq \oplus _{r=1}^k \bigl( -\cL_N^{(a_1, b_k)} \bigr) \leq \psi^{-1} \circ\bigl(- \cL^{(a_1,b_k)}_N \bigr)\circ \psi\,,$$
where the operator on the right is simply the self--adjoint operator
on $\oplus _{r=1}^k L^2(I_r,dm)$ with domain $\bigl\{ \psi^{-1}(f) :
f \in \cD\bigl(-\cL_N^{(a_1, b_k)} \bigr)\bigr\}$, mapping
$\psi^{-1}(f)$ into $\psi^{-1} \bigl(-\cL_N^{(a_1, b_k)} f\bigr)$.
At this point, the subadditivity follows from  the Lemma on page 270
of \cite{RS4} and property (5) on page 268 of \cite{RS4}.
\end{proof}

\subsection{Conclusion} We can now conclude stating the
Dirichlet--Neumann bracketing in our context:

\begin{Th}\label{DNB} (Dirichlet--Neumann bracketing). Let
$I=[a,b]$, let
  $$a=a_0<a_1< \cdots < a_{n-1}<  a_n=b$$ be a partition of the interval
  $I$ and set $I_r:=[a_r,a_{r+1}]$ for $r=0, \dots, n-1$.
Suppose that  $m:I\rightarrow \bbR$ is a nondecreasing function such
that
\begin{itemize}

\item[(i)]
 $m_{a_r}=0$ for all $r=0, \dots, n$,
\item[(ii)] $dm([a_r, a_r+\e] )
>0$ for all $r=0,\dots, n-1$ and $\e>0$,
\item[(iii)] $dm([a_r-\e, a_r])>0$ for all $r=1, \dots, n$ and  $\e>0$.
\end{itemize}
Then, for all $x \geq 0$ it holds
\begin{align}
& \cN^I _{m,D}(x) \leq \cN^I _{m, N} (x) \leq \cN^I_{m,D} (x)+2\,,\label{treno10}\\
& \cN ^I_{m,D} (x)\geq \sum _{i=0}^{n-1} \cN^{I_i} _{m,D}(x)\,\label{treno20}\\
& \cN ^I _{m,N}(x) \leq \sum _{i=0}^{n-1} \cN^{I_i}
_{m,N}(x)\,.\label{treno30}
\end{align}
\end{Th}
\begin{proof}
The bounds in \eqref{treno10} have been  obtained in
\eqref{agognatameta}. The inequalities \eqref{treno20} and
\eqref{treno30} follow from Lemma \ref{asdrubale}.
\end{proof}
As immediate consequence of \eqref{treno10} and \eqref{treno30} we
get a bound which will reveal very useful to derive
\eqref{kikokoala0} and \eqref{kikokoala00}:
\begin{Cor}\label{mela}
In the same setting of Theorem \ref{DNB} it holds $\cN^I _{m,D}(x)
\leq 2n +\sum _{i=0}^{n-1} \cN^{I_i} _{m,D}(x)$.
\end{Cor}

\section{Proof of Theorem  \ref{tacchino} }\label{persico}

  We first consider   how the eigenvalue counting functions change under
affine transformations:
\begin{Le}\label{similare} Let $m :\bbR\to \bbR$ be a nondecreasing c\`{a}dl\`{a}g
function.  Given the interval $I=[a,b]$,   suppose that $m_a=m_b=0$
and $dm \bigl( (a,a+\e)\bigr)>0$, $dm \bigl( (b-\e, b)\bigr)>0$ for
all $\e>0$. Given $\g,\b
>0$, set $J=[\g a, \g b]$ and define the  function $M: \bbR \rightarrow \bbR$ as
$M(x)= \g^{1/\beta } m(x/\g)$.  Then
\begin{equation}\label{polizia}
\cN ^I_{m,D/N} (x)= \cN ^J _{M, D/N} (x/ \g^{1+1/\beta } ) \,.
\end{equation}
\end{Le}
Trivially,  $M _{\g a} = M _{\g b}=0$ and  $dM \bigl( (\g a,\g
a+\e)\bigr)>0$, $dM \bigl( (\g b-\e, \g b)\bigr)>0$ for all $\e>0$
\begin{proof} For simplicity of notation we take $a=0$.
Suppose that $\l$ is an eigenvalue of the operator $-D_m D_x $ on
$[0,b]$ with Dirichlet b.c. at $0$ and $b$. This means that for a
nonzero function $F\in C(I)$ with $F(b)=0$  and a constant $c$ it
holds
\begin{equation}
 F(x)= cx -\l \int _0 ^x dy \int _{ [0,y)} dm (z) F(z) \,, \qquad \forall x \in I\,. \label{angelaaaa}
\end{equation}
Taking $X \in J$, the above identity implies that
\begin{multline}
F ( X/\g)= \frac{cX}{\g } -\l \int _0 ^{\frac{X}{\g}} dy   \int _{
[0,y)} dm (z) F(z) =\frac{c X}{\g}  -
\frac{\l}{\g} \int_0 ^X dY \int_{[0,\frac{Y}{\g})} dm (z) F(z)=\\
\frac{c X}{\g}  - \frac{\l }{\g^{1+1/\beta} } \int _0 ^X dY \int
_{[0,Y)} dM (Z) F(Z/\g)\,.
\end{multline}
Since trivially $F(X/\g)=0$ for $X=b\g$, the above identity implies
that $\l /\g^{1+1/\beta}$ is an eigenvalue of the operator $-D_{M}
D_x$  on $J$ with Dirichlet b.c.  and eigenfunction $F(\cdot /\g)$.
This implies \eqref{polizia} in the case of Dirichlet b.c. The
Neumann case is similar.
\end{proof}

\smallskip

We have now all the tools in order to prove Theorem  \ref{tacchino}:

\smallskip
\noindent  {\sl Proof of Theorem \ref{tacchino}}. Take  $m$ as in
Theorem  \ref{tacchino} and recall the notational convention stated
after the theorem. We first  prove
 \eqref{kikokoala1}, assuming without loss of generality that \eqref{kikokoala0} holds with $x_0=1$. By assumption,
with probability one, for any $n \in \bbN_+$ and any $k\in \bbN\,:\,
0\leq k \leq n$ it holds: (i) $ dm ( \{k/n\})=0$, (ii) $dm((k/n ,
k/n+\e))>0$ for all $\e>0$ if $k<n$, (iii) $dm((k/n-\e,k/n))>0$ for
all $\e>0$ if $k>0$. Below,  we assume that the realization of $m$
satisfies (i), (ii) and  (iii). This allows us to  apply the
Dirichlet--Neumann bracketing stated in Theorem  \ref{DNB} to the
non--overlapping subintervals $I_k :=[k/n, (k+1)/n]$, $k \in \{0,1,
\dots, n-1\}$.
 Due to the superadditivity (resp. subadditivity) of
the Dirichlet (resp. Neumann) eigenvalue counting functions (cf.
\eqref{treno20} and \eqref{treno30} in  Theorem \ref{DNB}), we get
for any $x \geq 0$ that $ N^{[0,1]} _{m ,D} (x)\geq \sum
_{k=0}^{n-1} N_{m,D} ^{I_k} (x)$, while $ N^{[0,1]} _{m ,N} (x)\leq
\sum _{k=0}^{n-1} N_{m,N} ^{I_k} (x)$. By taking the average over
$m$ and using that $m$ has   stationary increments we get that $
\bbE \cN ^{[0,1]} _{m,D} (x)\geq n \bbE \cN _{m,D} ^{[0,1/n]} (x)$
and $ \bbE \cN^{[0,1]} _{m,N} (x)\leq  n \bbE \cN _{m,N} ^{[0,1/n]}
(x)$.
  Using now the
scaling property of Lemma  \ref{similare}  with $\g=n,\b=\a$ and the
self--similarity of $m$, we conclude that
\begin{align}
& \bbE \cN ^{[0,1]} _{m,D} (x)\geq n  \bbE \cN  _{m,D}  ^{[0,1/n]}
(x) =
 n\bbE \cN   _{M, D} ^{[0,1] } (x/n ^{1+1/\alpha  } )
 =
  n\bbE \cN   _{m, D} ^{[0,1] } (x/n ^{1+1/\alpha  } )  \,,  \label{soldi1}\\
 & \bbE \cN ^{[0,1]} _{m,N} (x)\leq  n  \bbE \cN _{m,N}  ^{[0,1/n]} (x) =
  n \bbE \cN _{M,N} ^{[0,1] } (x/n ^{1+1/\alpha  } )=
 n\bbE \cN   _{m, N} ^{[0,1] } (x/n ^{1+1/\alpha  } )  \,,\label{soldi2}
\end{align}
where $M(x):= n^{1/\a} m(x/n)$. On the other hand, by
\eqref{treno10} of Theorem \ref{DNB}
\begin{equation}\label{cagnolino}  \bbE \cN ^{[0,1]} _{m,D} (x)\leq
\bbE \cN ^{[0,1]} _{m,N} (x)\leq \bbE \cN ^{[0,1]} _{m,D} (x)+2 \,.
\end{equation}
From the above estimates \eqref{soldi1}, \eqref{soldi2} and
\eqref{cagnolino},  we conclude that
\begin{equation}\label{bianchini}
\bbE \cN ^{[0,1]} _{m,D} (1) \leq n^{-1} \bbE \cN^{[0,1]} _{m,D} (
n^{1+1/\alpha } )\leq n^{-1} \bbE \cN^{[0,1]} _{m,N} ( n^{1+1/\alpha
} ) \leq \bbE \cN ^{[0,1]} _{m,N} (1) \leq \bbE \cN ^{[0,1]} _{m,D}
(1)+2 \,.
\end{equation}
%while
%\begin{equation}\label{lattino}
%\lim _{n\uparrow \infty} \bigl |n^{-2} \Phi^{[0,1]} _{D}
%(n^{1+1/\alpha } )- n^{-2} \Phi^{[0,1]} _{N} (n^{1+1/\alpha }
%)\bigr|=0\,.
%\end{equation}
We remark that \eqref{kikokoala0} with $x_0=1$ simply reads  $\bbE
\cN ^{[0,1]} _{m,D} (1)<\infty$.  Since the eigenvalue counting
functions are monotone, in the above estimate \eqref{bianchini}  we
can think of $n$ as any positive number larger than $1$. Then,
substituting $n^{1+1/\alpha }$ with $x$ we get \eqref{kikokoala1}.

\smallskip

\smallskip

In order to  prove \eqref{kikokoala2}, we first prove the joint
self--similarity of $m, m^{-1}$:   given $\g>0$, it holds
\begin{multline}\label{filmbis}
\bigl(  m(x), m^{-1} (y)\,:\, x ,y \geq 0 \bigr)\sim
 \bigl( \g^{1/\alpha  } m (x/\g), \g m^{-1} ( \g^{-1/\alpha  } y)\,:\, x,y  \geq 0 \bigr)\sim \\
\bigl(  \g m(x/\g^\a ),  \g^{\a} m^{-1} (x/\g)\,:\, x ,y \geq
0\bigr) \,.
\end{multline}
%\begin{equation}\label{filmbis}
%\bigl( m^{-1} (x)\,:\, x \geq 0 \bigr)\sim \bigl( \g m^{-1} (
%\g^{-1/\alpha  } x)\,:\, x \geq 0  \bigr)\sim \bigl( \g^{\a}
%m^{-1} (x/\g)\,:\, x \geq 0\bigr) \,.
%\end{equation}
To check the above claim, first we observe that for each $x \geq 0$
it holds
\begin{equation}\label{gioco}
\inf \left\{t \geq 0 \,:\,\g^{1/\alpha  } m(t/\g)>y\right\}= \g \inf
\left\{ t\geq 0\,:\, m(t)>
 \g^{-1/\alpha } y\right\}= \g m^{-1}(
\g^{-1/\alpha  } y)\,.
\end{equation}
On the other hand, by the self--similarity of $m$  and  by the
definition of the generalized inverse function, we get
\begin{equation}\label{gioco2}\left( \g^{1/\alpha } m(x/\g), \inf \left\{t \geq 0 \,:\,\g^{1/\alpha  } m(t/\g)>y\right\}\,:\, x,y \geq0 \right)\sim
\left(m(x),m^{-1} (y)\,:\, x, y \geq 0 \right)\,.
\end{equation}
The first identity in \eqref{filmbis} follows from \eqref{gioco} and
\eqref{gioco2}. The second identity follows by replacing
$\g^{1/\alpha }$ with $\g$. This concludes the proof of
\eqref{filmbis}.

\smallskip

 Recall the convention established after
\eqref{kikokoala2}.  We already know that $d m^{-1}$ is a continuous
function a.s., hence a.s. it holds (P1) $d m^{-1}\bigl( \{m(k/n)\}
\bigr)=0$ for all $n \in \bbN$  and $k \in \bbN: 0 \leq k \leq n$.
By identity \eqref{matriosca} $m^{-1}(x)=m^{-1}(y)$ if and only if
$x,y \in \bigl[m (z_i-), m(z_i) \bigr]$ for some jump point $z_i$ of
$m$.
 Since by
property (iv) in Theorem \ref{tacchino}   $m(k/n) $ is not a jump
point for $m$  a.s. (with $k,n$ as above), the following properties
hold a.s.: (P2) $dm^{-1} \bigl( (m(k/n) , m(k/n)+\e)\bigr)>0$ for
all $\e>0$ if $0\leq k<n$ and (P3) $dm^{-1} \bigl( (m(k/n)-\e ,
m(k/n))\bigr)>0$ for all $\e>0$ if $0<k\leq n$.
 In what follows we assume that the realization of $m$ satisfies the
 properties (P1), (P2) and (P3). This allows us to apply the
 Dirichlet--Neumann bracketing to the measure $dm^{-1}$ and to the
  non--overlapping subintervals $I_k=[m(k/n),
m((k+1)/n)]$, $k\in \{0,1,\dots, n-1\}$. We point out that
  the measure  $dm^{-1}$  restricted to  each subinterval $I_k$ is univocally determined by
the values $\{m(x)-m(k/n)  \,:\, x \in [k/n, (k+1)/n]$. The fact
that $m$ has  stationary increments, allows to conclude that the
random functions $N_{m^{-1} ,D/N} ^{I_k} (\cdot)$ are identically
distributed.

We observe now that \eqref{filmbis} with $\g= n^{1/\a}$ implies that
\begin{equation}\label{tiffany_maggio}
\bigl(  m(x), m^{-1} (y)\,:\, x ,y \geq 0 \bigr)\sim \bigl( n^{1/\a}
m(x/n ),  n  m^{-1} (x/n^{1/\a})\,:\, x ,y \geq 0\bigr) \,.
\end{equation}
Then, using the Dirichlet--Neumann, Lemma \ref{similare} with
$\b=1/\a$ and $\g=n^{1/\a}$ and the joint self--similarity
\eqref{tiffany_maggio},we conclude that
\begin{align}
& \bbE \cN ^{[0,m(1)]} _{m^{-1} ,D} (x)\geq n  \bbE \cN  _{m^{-1},D}
^{[0,m(1/n)]} (x) =
 n\bbE \cN   _{M, D} ^{[0,n^{1/\a}m(1/n) ] } (x/n ^{1+1/\alpha   } )
 =
  n\bbE \cN   _{m^{-1}, D} ^{[0,m(1)] } (x/n ^{1+1/\alpha  } )  \,,  \label{soldi1bis}\\
 & \bbE \cN ^{[0,m(1)]} _{m^{-1},N} (x)\leq  n  \bbE \cN _{m^{-1},N}  ^{[0,1/n]} (x) =
  n \bbE \cN _{M,N} ^{[0 ,n^{1/\a}m(1/n)   ] } (x/n ^{1+1/\alpha  } )=
 n\bbE \cN   _{m^{-1}, N} ^{[0,m(1)] } (x/n ^{1+1/\alpha  } )  \,,\label{soldi2bis}
\end{align}
where now $M(x)= n m^{-1} (x/ n^{1/\a})$. Note that
\eqref{soldi1bis} and \eqref{soldi2bis} have the same structure of
\eqref{soldi1} and \eqref{soldi2}, respectively. The conclusion then
follows  the same arguments used for \eqref{kikokoala1}. \qedhere
\qed

%%%%%%%%%%%%%%%%%%%%%%%%%%%%%%%%%%%%%%%%%
%%%%%%%%%%%%%%%%%%%%%%%%%%%%%%%%%%%%%%%%%
%%%%%%%%%%%%%%%%%%%%%%%%%%%%%%%%%%%%%%%%%

\section{Proof of Theorem \ref{affinita}}\label{affinita_proof}

As already mentioned in the Introduction,  the proof of  Theorem
\ref{affinita} is based on a special coupling introduced in
\cite{FIN} (and very similar to the coupling of \cite{KK} for the
random barrier model). If $\t(x)$ is itself the $\a$--stable law
with Laplace transform $ \bbE \bigl[e^{-\l \t(x) }\bigr] =
e^{-\l^\a}$, this coupling is very simple since it is enough to
define, for each realization of $V$ and for all  $n\geq 1$, the
random variables $\t_n(x)$'s as \begin{equation}\label{guerra}
 \t_n
(x) = n^{1/\a} \Big[ V\bigl(x+\frac{1}{n}\bigr)-V(x) \Big]\,, \qquad
\forall x \in \bbZ_n\,.
\end{equation}
Due to \eqref{salutare} and the fact that $V$ has independent
increments, one easily derives that the $V$--dependent random field
$\{ \t_n(x) \,:\, x \in \bbZ_n \}$ has the same law of $\{ \t(nx)
\,:\, x \in \bbZ_n \}$. In the general case one proceeds as follows.
Define a function $G:[0,\infty) \rightarrow [0,\infty)$ such that
$$\cP ( V(1)>G(x) )= \bbP ( \t(0)>x)\,, \qquad \forall x \geq 0
\,.$$ (Recall that $V$ is defined on the probability space $(\Xi,
\cF,\cP)$.) The above function $G$ is well defined since $V(1)$ has
continuous distribution, $G$ is right continuous and nondecreasing.
Then the generalized inverse function
$$G^{-1} (t) = \inf \{ x \geq 0 \,:\, G(x) >t \}$$
is nondecreasing and right continuous. Finally, set
\begin{equation}\label{spesa}
\t_n (x)= G^{-1} \left( n^\frac{1}{\a} \left[V\bigl(
x+\frac{1}{n}\bigr)- V(x) \right] \right) \,, \qquad x \in \bbZ_n\,.
\end{equation}
It is trivial to check that the $V$--dependent random field $\{
\t_n(x) \,:\, x \in \bbZ_n \}$ has the same law of $\{ \t(nx) \,:\,
x \in \bbZ_n \}$. Indeed, since  $V$ has independent and stationary
increments one obtains that the $\t_n(x)$'s are i.i.d., while since
$n^{\frac{1}{\a} }\left( V(x+\frac{1}{n})-V(x)\right)$ and $V(1)$
have the same law, one obtains that
$$\cP( \t_n (x) >t)= \cP( G^{-1} (V(1))
>t)=\cP( V(1)>G(t) )= \bbP( \t(nx )
>t)\,, \qquad \forall t \geq 0\,.$$
%Finally, we define
%$$ C(t)= 1/ \left[ \inf\{ s\geq 0\,:\, \bbP( \t(0) >s ) \leq 1/t \} \right]\,.$$
We point out that the coupling obtained by this general method does
not lead to \eqref{guerra} in the case that $\t(x)$ is itself the
$\a$--stable law with Laplace transform $ \bbE \bigl[e^{-\l \t(x)
}\bigr] = e^{-\l^\a}$.

\smallskip

%$$ g_n (x) = C(n) G^{-1} ( n^{\frac{1}{\a} } x )    \,, \qquad x\geq
%0\,,$$ where
%$$ C(n):= 1/  \inf \left\{ t\geq 0\,:\, \bbP(\t(0) >t) \leq 1/n
%)\right\}\,.
%$$

\subsection{Proof of Point (i)}
Let us  keep  definition \eqref{spesa}.
 For any $n\geq 1$ we
introduce the generalized
 trap model $\{\tilde X^{(n)}(t)\}_{t\geq
0}$ on $\bbZ_n$ with jump rates
$$c_n (x,y)=
\begin{cases} \g^2  L_2(n)  n^{1+\frac{1}{\a} } \t_n (x) ^{-1+a} \t_n (y)^{a} & \text{ if }
|x-y|=1/n\,\\
0 & \text{ otherwise}\,,
\end{cases}
$$
where  $ \g= \bbE ( \t(x)^{-a}) $.
%$$ S_n (k/n) =
%n^{-1}  \sum _{j=0} ^{k-1} \t_n (\frac{j}{n}) ^{-a} \t_n (
%\frac{j+1}{n}) ^{-a}
%$$
The above jump rates can be written as $c_n(x,y)= 1/ H_n(x) U_n
(x\lor  y)$ for $|x-y|=1/n$ by taking
$$
\begin{cases}
 U_n(x)=\g^{-2} n^{-1}  \t_n (x-\frac{1}{n})^{-a} \t_n (x)^{-a} \,\\
  H_n(x)=  L_2(n)^{-1} n^{-\frac{1}{\a}}  \t_n (x)  \,.
  \end{cases}
$$
  Note that in all cases both $U_n$ and $H_n$ are functions
of the $\a$--stable subordinator  $V$.

Then the following holds
\begin{Le}\label{normale}
Let $m_n$ be defined as in \eqref{treno3} by means of the above
functions $U_n,H_n$. Then for almost any realization of the
$\a$--stable subordinator $V$,  $ \ell_n \rightarrow 1$ and the
measures  $dm_n$ weakly converge to the measure $d V_*$ (recall
definition \eqref{percome}).
\end{Le}
\begin{proof}
% kuka
Due to our definition \eqref{treno1} we have
$$ S_n\Big(\frac{k}{n}\Big)= \frac{1}{n} \sum _{j=1}^{k} \g^{-2} \t_n \Big(\frac{j-1}{n}\Big)^{-a}
 \t_n \Big(\frac{j}{n}
\Big)^{-a}\,, \qquad 0 \leq k \leq n\,,
$$
with the convention that the sum in the r.h.s. is zero if $k=0$. If
$a=0$ trivially  $\g=1$ and $S(k/n)=k/n$.
% and
%$$ dm_n=  \sum _{k=0}^n L_3(n) ^{-1} n^{-\frac{1}{\a} }
%\t_ n (k/n)\d _{k/n}\,.
%$$
%Reasoning as in \cite{FIN} \club, we obtain that $dm_n \rightarrow d
%\tilde V$, for almost all realization of $V$.
If $a>0$ we can apply the strong law of large numbers for triangular
arrays. Indeed, all addenda have the same law and they are
independent if they are not consecutive, moreover they have bounded
moments of all orders since  $\t(x)$ is bounded from below by a
positive constant a.s. (this assumption is used only here and could
be weakened in order to assure the validity of the strong LLN).
 Due
to the choice of $\g$ we have that $ \g^{-2} \t_n
\bigl(\frac{j-1}{n}\bigr)^{-a} \t_n \bigl (\frac{j}{n} \bigr)^{-a}$
has mean $1$. By the strong law of large number we conclude that for
a.a. $V$ it holds
% Then, reasoning as in the proof of Lemma
%5.2 in \cite{BC1}, and in particular using the fact that $S(k/n)$ is
%an average of an ergodic sequence of bounded random variables which
%are independent if their indexes are not nearest--neighbor points,
% one gets that
$ \lim_{n\uparrow \infty}  S\bigl( \lfloor xn \rfloor /n\bigr) =x$
for all $ x \geq 0$.  This proves in particular that $\ell_n:=
S_n(1)\rightarrow 1$. It remains to prove that for all $f \in C_c
(\bbR)$ it holds
\begin{equation}\label{cif} \lim _{n\uparrow \infty}  \sum_{k=0}^n
f(S_n(k/n))H_n(k/n) = \int _0 ^1 f(s) d V_* (s) \,.
\end{equation}
This limit can be obtained by reasoning as in the proof of
Proposition 5.1 in \cite{BC1}, or can be derived by Proposition 5.1
in \cite{BC1} itself together with the fact that $\cP$ a.s. $V$ has
no jump at $0,1$. To this aim one has to observe that the constant
$c_\e$ (where $\e=1/n$) in \cite{FIN} and \cite{BC1}[eq. (49)]
equals our quantity $1/h(n)=1/\left(n^{1/\a} L_2 (n)\right)$ (recall
the definitions preceding Theorem \ref{affinita}). In particular,
$H_n (k/n)= c_{1/n} \t _n (k/n)$.\qedhere
\end{proof}
%%%%%%%%%%%%%%%%%%%%%%%%%%%%%%%%%%%%%%%%
Due to the above result, Point  (i)  in   Theorem \ref{affinita}
follows easily from Theorem \ref{speriamo} and the fact that the
random fields $  \{ \t_n(x) \,:\, x \in \bbZ_n \}$ and $\{ \t(nx)
\,:\, x \in \bbZ_n \}$ have the same law for all  $n\geq 1$.
%where the symbol $\sim$ means that the two processes have the same
%law.

\subsection{Proof of Point (ii)}
Point (i) can be proved in a similar and simpler way. In this case,
we define $\t_n(x)$ as in \eqref{guerra} and  we consider the
generalized
 trap model $\{\tilde X^{(n)}(t)\}_{t\geq
0}$ on $\bbZ_n$ with jump rates
$$c_n (x,y)=
\begin{cases}     n^{1+\frac{1}{\a} } \t_n (x) ^{-1}   & \text{ if }
|x-y|=1/n\,\\
0 & \text{ otherwise}\,,
\end{cases}
$$
  with associated functions $$U_n(x)=1/n\,, \qquad  H_n(x)=
n^{-\frac{1}{\a} }\t_n (x) = V(x+1/n)-V(x)=:\D_n V(x)\,.$$ By this
choice,  $ dm_n = \sum _{k=0}^n \d_{k/n} \D_n V(k/n)$. Trivially, $
\ell_n =1$ and $dm_n\rightarrow d  V_*$ for all realizations of $V$
giving zero mass to the extreme points $0$ and $1$. Since this event
takes place $\cP$--almost surely, the proof of part (ii) is
concluded.

\subsection{Proof of Point (iii)}
 Part (iii) of   Theorem \ref{affinita}   (i.e.
\eqref{leoncino}) follows  from Theorem \ref{tacchino} and Lemma
\ref{dimitri} below.
% and   the self--similarity of the $\a$--stable process $V$:
% by means
%of the Dirichlet--Neumann bracketing for the eigenvalues of
%boundary--value problem.
% \cite{CH1}, \cite{CW},  \cite{Fr}  .
 The self--similarity  of $V$ is the following:
 for each $\g>0$ it holds
\begin{equation}\label{film}
\bigl(V(x)\,,\, x\in \bbR\bigr) \sim \bigl( \g^{\frac{1}{\a} } V(
x/\g )\,:\, x \in \bbR\bigr) \,.
\end{equation}
Indeed, both processes are c\`{a}dl\`{a}g, take value $0$ at the
origin and  have independent increments with the same law due to
\eqref{salutare}.

\begin{Le}\label{dimitri}
Taking $m=V$, the bound \eqref{kikokoala0} is satisfied.
\end{Le}
\begin{proof}
Using the notation of Section \ref{persico}, we denote by
$\cN_{V,D}^{[0,1]}(1)$ the number of   eigenvalues not larger than
$1$ of the operator $- D_V D_x$ on $[0,1]$ with Dirichlet boundary
conditions. We assume that $V$ has no jump at $0,1$ (this happens
$\cP$--a.s.).  We recall that $V$ can be obtained by means of the
identity $dV= \sum _{j\in J} x_j \d_{v_j}$, where the random set
$\xi=\{(x_j,v_j): j\in J\}$ is the realization of a inhomogeneous
Poisson point process    on $\bbR\times \bbR_+$  with intensity $c
v^{-1-\a} dx dv$, for a suitable positive constant $c$. In order to
distinguish between the contribution of big jumps and not big jumps
it is convenient to work with two independent inhomogeneous Poisson
point processes $\xi ^{(1)}$ and $\xi ^{(2) }$ on $\bbR\times
\bbR_+$ with intensity $c v^{-1-\a} \bbI (v\leq 1/2) dx dv$ and  $c
v^{-1-\a} \bbI(v>1/2) dx dv$. We write $\xi ^{(1)}= \{ (x_j,v_j): j
\in J_1\}$ and  $\xi ^{(2)}= \{ (x_j,v_j): j \in J_2\}$. The above
point process $\xi$ can be defined as $\xi=\xi ^{(1)}\cup \xi
^{(2)}$. Moreover, a.s. it holds $\xi ^{(1)} \cap \xi ^{(2)}
=\emptyset$ (this fact will be understood in what follows). By the
Master Formula (cf. Proposition (1.10) in \cite{RY}), it holds
\begin{align}
& \bbE\Big[\sum _{ j\in J_1\,:\,  x_j \in [0,1] } v_j \Big] =c \int
_0^1 dx \int _0 ^{1/2} dv\,  v^{-\a} < \infty\,,\label{adavide1} \\
&  \bbE\Big[\sharp \{j\in J_2\,:\, x_j \in [0,1]\}\Big] =c \int _0^1
dx \int _{1/2} ^\infty dv\, v^{-1-\a} < \infty \,.\label{adavide2}
\end{align}
We label in increasing order the points in  $\{x_j \,:\, j\in
J_2\,,\; x_j \in [0,1]\}$ as $y_1<y_2< \dots < y_N$ (note that the
set is finite due to \eqref{adavide2}).

Given $\d\in (0,1/8)$, we take   $\e\in (0,1) $ small enough that

\begin{itemize}

\item[(i)] the intervals $(y_i-\e, y_i+\e)$ are included in $(0,1)$ and  do not intersect as $i$ varies
from $1$ to $N$,

\item[(ii)]  for all $i:1\leq i\leq N$, it holds $\sum _{j\in J_1: x_j \in (y_i-\e, y_i+\e)} v_j <\d
$,

\item[(iii)] for all $i: 1\leq i \leq N$, the points $y_i-\e$ and
$y_i+\e$ do not belong to $\{x_j:j\in J_1\}$.
\end{itemize}

Defining  $V^{(1)}(t)= \sum _{j\in J_1\,:\, x_j \leq t} v_j$, the
last condition (iii) can be stated as follows: for all $i: 1\leq i
\leq N$, the points $y_i-\e$ and $y_i+\e$  are not jump points for
$V^{(1)}$.

 By construction the
function $V^{(1)}$ has jumps not larger than $1/2$. In particular,
all the intervals $ A_0= (0, y_1-\e)$, $A_1=(y_1+\e, y_2-\e)$,
$A_2=(y_2+\e, y_3-\e)$,..., $A_{N-1} = (y_{N-1}+\e, y_{N}-\e) $,
$A_N=(y_N+\e,1)$ can be partitioned in
 subintervals  such that, on each
subinterval,  the function  $V^{(1)}$ has increment in $[1/2, 1)$
and has no jump at the border (recall property (iii) above). As a
consequence, the total number  $R$  of subintervals is bounded by $2
V^{(1)} (1)$, which has finite expectation due to \eqref{adavide1}.
By the bound \eqref{materazzi} in Lemma \ref{volare}, we get that
the operator $-D_V D_x$ on any subinterval with Dirichlet boundary
conditions has no eigenvalues smaller than $2$. This observation and
Corollary \ref{mela} imply that
\begin{equation}\label{profitto}
 \cN ^{[0,1]} _{D, V} (1) \leq 2R + \sum _{i=1}^N \cN
^{[y_i-\e,y_i+\e]} _{D, V} (1)\,.
\end{equation}

\noindent {\bf Claim}: {\sl For each $i:1\leq i \leq N$ it holds
$\cN ^{[y_i-\e,y_i+\e]} _{D, V} (1)\leq 1$.}

\smallskip

\noindent
 {\sl Proof of the
claim.} We reason by contradiction supposing that $f_1$ and $f_2$
are eigenfunctions of the Dirichlet operator $-D_V D_x$ on
$U=[y_i-\e, y_i+\e]$,  whose corresponding eigenvalues $\l_1$ and
$\l_2$ satisfy  $0<\l_1<\l_2\leq 1$. We can take $f_1$ and $f_2$
continuous on $U$, satisfying $\int _U f_j ^2 (x) dV(x)=1$ and
\begin{equation}\label{girotondo} |f_j(x)- f_j(y_i) |\leq \sqrt{|x-y_i|}
\,,\qquad x \in U
\end{equation}
for $j=1,2$. Indeed, recall that $|f(x)-f(y)|^2 \leq q_D(f)  |x-y|$ for any $f \in  Q(q_D)$ in Section \ref{sec_dnb}.    %  Without loss, we assume that   $1=\int _U f_j ^2 (x) dV(x)$ for $j=1,2$.
  Calling $\D= dV (\{y_i\})$,  \eqref{girotondo}
and property (ii)   imply that
$$ 1= \int _U f_j ^2 (x) dV(x) \leq \D f_j^2 (y_i)+ \d( |f_j(y_i)|+ \sqrt{\e} )^2 \leq \D f_j^2 (y_i)+ 2\d f^2_j(y_i) + 2 \d
\e\,.
$$
In particular, we  get $ f_j ^2 (y_i) \geq (1- 2 \d\e)/(\D+ 2\d)$.
Due to our choice of the constants, $ 1-2\d\e \geq 1- 2 (1/8)=3/4$,
while $\D+2\d \leq\D+ 1/4< (3/2)\D$ (recall that $\D >1/2$). Hence,
we get that $\D f_j ^2 (y_i) \geq 1/2$. On the other hand, using the
orthogonality between $f_1$ and $f_2$, it must be
\begin{equation}\label{ham1} 1/4\leq  \bigl| \D f_1 (y_i) f_2
(y_i)\bigr| =\bigl| \int _{U\setminus\{y_i\} } f_1 (x) f_2 (x)
dV(x)\bigr| \leq \d ( |f_1 (y_i)|+\sqrt{\e}) ( |f_2
(y_i)|+\sqrt{\e}) \,.
 \end{equation}
Since by construction $\e\leq 2\leq  \D f_j ^2 (y_i)$ and $\D> 1/2$
we can bound \begin{equation}\label{ham2} |f_j (y_i)|+\sqrt{\e}\leq
\sqrt{2} \sqrt{\D} |f_j (y_i)|+\sqrt{\e}\leq (1+\sqrt{2})\sqrt{\D}
|f_j (y_i)|\,.
\end{equation}
Combining \eqref{ham1} and \eqref{ham2}, we conclude that $$1/4\leq
\bigl| \D  f_1 (y_i) f_2 (y_i)\bigr|\leq (1+ \sqrt{2})^2 \d \bigl|
\D f_1 (y_i) f_2 (y_i)\bigr|\,,$$ in contrast with the bound $\d
<1/8$. \qed

 Applying the above claim to \eqref{profitto} we conclude that
$\cN ^{[0,1]} _{D, V} (1) \leq 2R +N$. We have already observed that
$R$ has finite expectation. The same trivially holds also for $N$
due to \eqref{adavide2}.
\end{proof}

\section{Proof of Theorem \ref{affinitabis}}\label{primopiano}
Recall the definition of $\cT_n$ given in the previous section.
Given a realization of $V$, for each $n\geq 1$ we consider the
continuous--time nearest--neighbor random walk $\tilde X^{(n)}$  on
$\bbZ_n$ with jump rates \begin{equation}\label{modena} c_n (x,y)=
\begin{cases} L_2 (n) n^{1+\frac{1}{\a} } \t_n (x\lor y)^{-1} & \text{ if }
|x-y|=1/n\,,\\
0 & \text{ otherwise}\,.
\end{cases}
\end{equation}
The rates $c_n(x,y)$ for $|x-y|=1/n$  can be written as $c_n(x,y)=
1/\bigl[H_n(x,y) U_n(x \lor  y) \bigr]$,  where $H_n(x)=1/n$ and
$U_n (x)= L_2(n)^{-1} n^{-\frac{1}{\a} }\t_n (x)$. To the above
random walk we associate the measure $dm_n$ defined in
\eqref{treno3}.

\subsection{Proof of Point (i)}
Let us show  that $dm_n$ weakly converges to $d( V^{-1} )_*$ (recall
\eqref{percome}). We point out that in \cite{KK} a similar result is
proved, but the definition given in \cite{KK} of the analogous of $d
m_n$ is different, hence that proof cannot be adapted to our case.
In order to prove the weak convergence  of $dm_n$ to $d( V^{-1})_*
$, we use some results and ideas developed in Section 3 of
\cite{FIN}. Recall   that the constant $c_\e$  of \cite{FIN} equals
our quantity $1/h(n)=1/\left(n^{1/\a} L_2 (n)\right)$ if $\e=1/n$ .
Given $n\geq 1$ and $x>0$ we define
$$ g_n(x)=  \bigl(L_2(n) n^{\frac{1}{\a} }\bigr)^{-1}  G^{-1} (
n^{\frac{1}{\a} } x)\,.
$$
We point out that $g_n$ coincides with the function $g_\e$ defined
in \cite{FIN}[(3.12)] if $\e=1/n$.
 As stated in Lemma 3.1 of \cite{FIN} it holds
%\begin{equation}\label{mandarino=}
$ g_n(x) \rightarrow x $ as $n\rightarrow \infty$ for all $x >0$.
 Since $g_n$ is nondecreasing, we conclude that
\begin{equation}\label{mandarinobis}
 g_n(x_n) \rightarrow x \text{ as  }n\rightarrow \infty\,, \qquad \forall x >0, \; \forall \{x_n\}_{n\geq 1}: x_n >0 \,, \;x_n \rightarrow x \,.
 \end{equation}
  As stated in Lemma 3.2 of \cite{FIN}, for
any $\d'>0$ there exist positive constants $C'$ and $C''$ such that
\begin{equation}\label{pistacchioso}
g_n(x) \leq C' x^{1-\d'} \text{ for } n^{-\frac{1}{\a}} \leq x \leq
1 \text{ and } n \geq  C''\,.
\end{equation}
Since  $U_n(x)= g_n \bigl( V(x+1/n)-V(x)\bigr)$, we can write
%$$ \int _{\bbR} f(x) dm_n(x)= \frac{1}{n} \sum _{k=0}^n f\bigl(
\begin{equation}\label{sonnissimo}
 S_n \bigl(k/n\bigr) = \sum _{j=0}^{k-1}
 g_n \left( V\bigl( (k+1)/n \bigr)- V\bigl( k/n
 \bigr)\right)\,.
\end{equation}
\begin{Le}\label{franchi1}
For $\cP$--almost all $V$ it holds
\begin{equation}\label{pc}
\lim _{n \uparrow \infty} \max_{0\leq k \leq n} \bigl | S_n (k/n) -
V(k/n) \bigr|=0\,.
\end{equation}
\end{Le}
\begin{proof}
We recall that $V$ can be obtained by means of the identity $dV=
\sum _{j\in J} x_j \d_{v_j}$, where the random set $\xi=\{(x_j,v_j):
j\in J\}$ is the realization of a inhomogeneous Poisson point
process    on $\bbR\times \bbR_+$  with intensity $c v^{-1-\a} dx
dv$, for a suitable positive constant $c$.
  %$\cP$ a.s. the set $\{x_j\}$ does
%not contain any rational point.
 Given $y>0$, let us define
\begin{align*}
& J_{n,y}:=\{r\in\{0,1, \dots,n-1\} \,:\, V((r+1)/n)- V(r/n) \geq y
\}\,,\\
& J_y:=\{ j\in J\,:\, v_j \geq y\,,\;x_j \in [0,1] \} \,.
\end{align*}
%Since $E(|J_y|)=c\int_{[0,1]} du \int _y ^\infty w^{-\a} dw <
%\infty$, for almost all $V$
Note that  the set $J_y$ is always finite.
%Given $\e>0$, choose $\d>0$ small enough such that
%$$ |V(1)- \sum _{i \in J_\d} w_j| \leq \e \,
%$$
Reasoning as in the Proof of Proposition 3.1 in \cite{FIN}, and in
particular using also  \eqref{pistacchioso},  one obtains   for
$\cP$--a.a. $V$ that
 \begin{equation}\label{frutta}
\limsup _{n \uparrow \infty} \sum _{r: 0 \leq r <n\,, r \not \in
J_{n,\d}  }
 g_n \left( V\bigl( (r+1)/n \bigr)- V\bigl( r/n
 \bigr)\right)=0 \,, \qquad \forall \d >0\,.
 \end{equation}

\smallskip

We claim that, given $\d >0$, for a.a. $V$ it holds
\begin{equation}\label{pietrino}
J_{n, \d}= \bigl\{ r \in \{0, 1, \dots, n-1\}\,:\, \exists j \in
J_\d \text{ such that } x_j \in (r/n, (r+1)/n] \bigr\}
\end{equation}
eventually in $n$. Let us suppose that \eqref{pietrino} is not
satisfied. Since the set in the r.h.s. is trivially included in
$J_{n,\d}$,  there exists a sequence of integers $r_n$ with $0 \leq
r_n < n$ such that $a_n:=V( (r_n+1)/n)- V(r_n/n) \geq \d$ while $
v_j < \d$ for all $x_j \in (r_n, (r_n+1)/n]$. We introduce the
c\`{a}dl\`{a}g
 function $\bar V(t)= \sum _{j\in J: x_j \leq t } v_j\bbI( v_j < \d)
$ and we note that, if $\forall  j \in J $ with $x_j \in (r_n/n,
(r_n+1)/n]$ it holds $v_j < \d$,  then $a_n=\bar V( (r_n+1)/n)- \bar
V(r_n/n)$.  At cost to take a subsequence, we can suppose that $r_n
/n$ converges to some point $x$. It follows then that $\bar
V(x+)-\bar V(x-)\geq \d$, in contradiction with the fact that $\bar
V$ has only jumps smaller than $\d$. This concludes the proof of our
claim.

\smallskip

Due to the above claim and due to  \eqref{mandarinobis}, we conclude
that a.s., given $\d>0$, it holds
\begin{equation}\label{succo} \lim _{n \uparrow \infty} \sup_{1\leq k \leq n }\Big|
\sum _{r \in J_{n,\d}, r <k }
 g_n \left( V\bigl( (r+1)/n \bigr)- V\bigl( r/n
 \bigr)\right)-
 \sum _{j \in J_\d: x_j \leq k/n } v_j\Big| =0\,.
 \end{equation}
Combining \eqref{succo} and \eqref{frutta}, we conclude that for any
$\e>0$ one can fix a.s. $\d>0$ small enough such that
\begin{equation}\label{anno0}
 \max_{ 0\leq k \leq n } \bigl|S(k/n)- \sum _{j
\in J_\d: x_j \leq k/n } v_j \bigr|\leq \e \end{equation} for $n$
large enough. On the other hand, a.s. one can fix $\d$ small enough
that $\sum _{j\in J_\d: x_j \in [0,1]} v_j $ is bounded by $\e$.
This last bound and \eqref{anno0} imply \eqref{pc}.
\end{proof}

\begin{Le}\label{franchi2}
For $\cP$--almost  all $V$ and
% without jumps in $0$ or $1$, and in particular for $\cP$--almost all $V$,
  for any function $f \in C_c (\bbR)$
it holds
\begin{equation}\label{latino} \lim _{n\uparrow \infty}\frac{1}{n} \sum
_{k=0}^{n} f\left( S_n (k/n) \right) = \int _{[0,V(1)]} f(x) d
V^{-1}(x) \,.
\end{equation}
\end{Le}
%Since $V^{-1}$ is a continuous function, the measure $dV^{-1}$ has
%no atom. Hence in the r.h.s. of \eqref{latino} we do not need
\begin{proof} Since $f$ is uniformly continuous, by Lemma \ref{franchi1} it
is enough to prove  \eqref{latino} with $S_n (k/n)$ replaced by
$V(k/n)$. Approximating $f$ by stepwise functions with jumps on
rational points, it is enough to prove that, fixed  $t \in \bbQ$,
for $\cP$--a.a. $V$ the limit \eqref{latino} holds with $S_n (k/n)$
replaced by $V(k/n)$ and with $f(x)= \bbI(x \leq t)$. This last
check is immediate.
\end{proof}
We have now all the tools in order to prove Point (i) of Theorem
\ref{affinitabis}. Indeed,  by Lemma \ref{franchi1} $\ell_n=
S_n(1)\rightarrow  V(1)$ $\cP$--a.s. Moreover, by Lemma
\ref{franchi2} the measure $dm_n$ defined in \eqref{treno3} weakly
converges to the measure $d (V^{-1})_*$. In order to get Point (i)
of Theorem \ref{affinitabis} it is enough to apply Theorem
\ref{speriamo}.

\subsection{Proof of Point (ii)}

 If $\bbE( e^{-\l \t(x)})=
e^{-\l^\a} $ one can replace $L_2 (n)$ with $1$ in \eqref{modena}
and in the above  definition of $U_n(x)$, and one can define
$\t_n(x)$ directly by means of \eqref{guerra}. In this case,
definition \eqref{treno1} gives $S_n (k/n)=V\bigl((k+1)/n\bigr)$ and
therefore $dm_n= \frac{1}{n} \sum _{k=1}^{n+1} \d_{V(k/n) }$.
%Approximating continuous functions by stepwise functions with jumps
%at rational points, it is simple to prove that, for almost any
%realization of $V$, it holds
%$$
%\int f(x) dm_n (x)=\frac{1}{n} \sum _{k=1}^{n+1} f\bigl( V(k/n)
%\bigr)\to \int _{[0,1]} f(V(x) ) dx \,,\qquad \forall f \in C_c
%(\bbR)\,.
%$$
%By a change of variable, the last integral equals $\int _{[0, V(1)
%]} f(y) V^{-1} (dy)$.
It is  simple to prove that  a.s. $dm_n$ weakly converges to
$dm:=d(V^{-1})_*$. Hence,  one gets that the assumptions  of Theorem
\ref{speriamo} are fulfilled with $\ell_n=V\bigl( (n+1)/n)$,
$\ell=V(1)$ and $dm=(V^{-1})_*$, for almost  all realization of $V$.
As a consequence, one derives Point (ii) in  Theorem
\ref{affinitabis}.

\subsection{Proof of Point (iii)}
The proof of point (iii) of Theorem \ref{affinitabis} follows  from
Theorem  \ref{tacchino} once we prove  \eqref{kikokoala00} with
$m=V$. As in the proof of Lemma \ref{dimitri} we denote by $0<y_1<
y_2< \cdots < y_N<1$ the points in $[0,1]$ where $V$ has a jump
larger than $1/2$ (note that  $V$ is continuous in $0$ and $1$
a.s.).
 We set
$a_i:= V(y_i-)$, $b_i= V(y_i)$ and remark that   the function
$V^{-1}$ is constant on $[a_i,b_i]$. Then we fix $\e>0$ (which is a
random number) such that  the following properties holds:
\begin{itemize}
\item[(i)] the intervals $U_i:=[a_i-\e,b_i+\e]$, $i=1,...,N$,  are disjoint
and included in $[0, V(1)]$,
\item[(ii)] $V$ has no jump at $a_i-\e$ and $b_i+\e$, for all $i=1
,\dots, N$,
\item[(iii)]  for all $i=1, \dots, N$,
\begin{equation}\label{tamburo}
(b_i-a_i +2\e) \bigl( V^{-1} (b_i+\e) - V^{-1}(a_i-\e) \bigr) \leq
1/2\,.\end{equation} \end{itemize} Note that, since $V^{-1}$ is
continuous a.s.  and flat on $U_i$, condition (iii) is satisfied for
$\e$ small enough. Moreover, due to condition (ii) it holds
$V^{-1}(x)< V^{-1}(y)<V^{-1}(z)$ if $y\in \{a_i-\e, b_i +\e\}$ and
$x<y<z$.

\smallskip
Let now $f$ be an eigenfunction of the operator $-D_{V^{-1}}D_x$ on
$U_i$ with Dirichlet boundary conditions. Writing $\l$ for the
associated eigenvalue,  by equation \eqref{nonnina} in Lemma
\ref{volare} it holds
$$ f(x)= \l \int _{U_i} G_{a_i-\e, b_i+\e}(x,y) f(y) d
V^{-1} (y) \,.
$$
Using that $\|G_{a_i-\e, b_i+\e}\|_\infty \leq b_i-a_i+2\e$ we get
\begin{equation}\label{esamini}|f(x)| \leq \l (b_i-a_i+2\e) \|f\|_\infty\bigl( V^{-1} (b_i+\e)
- V^{-1}(a_i-\e) \bigr)\,.
\end{equation}
Combining \eqref{tamburo} and \eqref{esamini} we conclude that $\l
\geq 2$. Hence $\cN ^{U_i} _{V^{-1}, D} (1) =0$. We now observe that
the set  $ W= [0, V(1) ] \setminus \cup _{i=1}^N U_i $ is the union
of $N+1$ intervals and its total length is smaller than $V^{(1)}(1)$
(see the proof of Lemma \ref{dimitri} for the definition of
$V^{(1)}$). It follows that we can partition  $W$ in at most $2
V^{(1) }(1)+N$ subintervals $A_r$  of length bounded by $1/2$. Since
the $dV^{-1}$--mass of any subinterval $A_r$ is bounded by the total
$d V^{-1}$--mass of $[0, V(1) ]$ (which is a.s. $1$), by the
estimate \eqref{materazzi} in Lemma \ref{volare} we get that  all
eigenvalues of the operator $-D_{V^{-1}} D_x$ restricted to any
subinterval $A_r$ (with Dirichlet b.c.) is at least $2$, hence  $\cN
^{A_r}_{V^{-1}, D}(1)=0$. We now apply  Corollary \ref{mela},
observing that we are in the same setting on Theorem \ref{DNB}
(recall that $V^{-1}$ is continuous a.s. and recall our condition
(ii), thus leading to (i)--(iii) in Theorem \ref{DNB}). By Corollary
\ref{mela},  we conclude that $\cN^{[0, V(1)]}_{V^{-1} , D} (1) \leq
V^{(1)}(1) +4 N$ a.s. As already observed in the proof of Lemma
\ref{dimitri}, both $V^{(1)}(1)$ and $N$ have finite expectation,
thus leading to \eqref{kikokoala00}.

\section{ The diffusive case: Proof of Propositions \ref{maschio1} and \ref{maschio2}
}\label{diffondo}

\subsection{Proof of Proposition \ref{maschio1}}
 We consider the diffusively rescaled random walk $X^{(n)}$on
 $\bbZ_n$ with jump rates
$$c_n (x,y)=
\begin{cases} \bbE( \t(0)^{-a} )^2 \bbE( \t(0) ) n^2  \t  (nx ) ^{-1+a} \t  (ny )^{a} & \text{ if }
|x-y|=1/n\,\\
0 & \text{ otherwise}\,.
\end{cases}
$$
%$$ S_n (k/n) =
%n^{-1}  \sum _{j=0} ^{k-1} \t_n (\frac{j}{n}) ^{-a} \t_n (
%\frac{j+1}{n}) ^{-a}
%$$
The above jump rates can be written as $c_n(x,y)= 1/ H_n(x) U_n
(x\lor  y)$ for $|x-y|=1/n$ by taking
$$
\begin{cases}
 U_n(x)= \bbE( \t(0)^{-a} )^{-2} n^{-1}  \t (nx-1 )^{-a} \t (n x)^{-a} \,\\
  H_n(x)= \bbE( \t(0) )^{-1}   n^{-1}   \t (n x)  \,.
  \end{cases}
$$
Due to our definition \eqref{treno1} we have
$$ S_n\bigl(k/n \bigr) = \frac{1}{n \bbE( \t(0) ^{-a} )^2 } \sum_{j=1}^k   \t(j-1)^{-a} \t(j)^{-a}\,, \qquad 0 \leq k \leq n\,.
$$
By the ergodic theorem and the assumption  $\bbE\bigl(
\t(0)^{-a}\bigr) <\infty$,  it holds $ \lim _{n\uparrow \infty} S_n
\bigl( \lfloor x n \rfloor /n)=x$ for all $x \geq 0$ (a.s.). In
particular, it holds $\ell_n= S_n(1) \rightarrow 1$. Since $\p^2
k^2$ is the $k$--th eigenvalue of $-\D$ with Dirichlet conditions
outside $(0,1)$, by Theorem \ref{speriamo}  it remains to prove
that, a.s., for all $f\in C_c ([0,\infty))$ it holds
\begin{equation}\label{ciuff}
\lim_{n\uparrow \infty} dm_n (f)= \lim_{n\uparrow \infty}
\sum_{k=0}^n f( S_n(k/n) ) H_n(k/n)= \int_0^1 f(s) ds \,.
\end{equation}
By the ergodic theorem and the assumption $\bbE\bigl( \t(0)
\bigr)<\infty$, the total mass of $dm_n$, i.e.$ \sum _{k=0}^n
H_n(k/n)$, converges to $1$ a.s. Hence,  by a standard approximation
argument with stepwise functions, it is enough to prove
\eqref{ciuff} for functions $f$ of the form $f=\bbI ([0,t) )$. By
the ergodic theorem a.s. it holds: for any $\e>0$ there exists a
random integer $n_0$ such that $ S_n(k/n)< t$ for all $ k \leq
(t-\e) n$ and $S_n(k/n)>t$ for all $ k \geq (t+\e)/n $. Therefore,
for $f$ as above and  $n \geq n_0$, we can bound
$$
\frac{1}{n \bbE( \t(0) )} \sum _{k\in \bbN: k \leq (t-\e) n}
\t(k)\leq dm_n (f) \leq\frac{1}{n \bbE( \t(0) )} \sum _{k\in \bbN: k
\leq (t+\e) n} \t(k)\,.
$$
Applying again the ergodic theorem,  it is immediate to conclude.

\subsection{Proof of Proposition \ref{maschio2} }
We sketch the proof since the technical steps are very easy and
similar to the ones discussed above.
 We consider the diffusively rescaled random walk $X^{(n)}$on
 $\bbZ_n$ with jump rates
$$ c_n (x,y)=
\begin{cases} n^2 \bbE(\t(0) )  \t (nx\lor n y)^{-1} & \text{ if }
|x-y|=1/n\,,\\
0 & \text{ otherwise}\,.
\end{cases}
$$
The rates $c_n(x,y)$ for $|x-y|=1/n$  can be written as $c_n(x,y)=
1/\bigl[H_n(x,y) U_n(x \lor  y) \bigr]$,  where $H_n(x)=1/n$ and
$U_n (x)=  \t(n x)/ n \bbE(\t(0) ) $. By the ergodic theorem and the
assumption $\bbE(\t(0))<\infty$, a.s. it holds $\lim_{n\uparrow
\infty} S_n ( \lfloor nx \rfloor)= x$ for all $x \geq 0$. In
particular, a.s. $S_n(n) \rightarrow 1$ and  $$ \lim _{n\uparrow
\infty} dm_n(f) = \lim _{n\uparrow \infty} \frac{1}{n} \sum _{k=0}^n
f\bigl( S_n (k/n) \bigr) = \int_0^1 f(x) dx\,,
$$
for all $f \in C_c ([0,\infty)).$ At this point it is enough to
apply Theorem \ref{speriamo}.
\bigskip

\bigskip

\noindent {\bf Acknowledgements}. I  thank Jean--Christophe Mourrat
and  Eugenio Montefusco   for useful discussions, and an anonymous
referee for      useful suggestions. I  acknowledge the financial
support of the European Research Council through the ``Advanced
Grant''  PTRELSS 228032.  Part of this work has been done just after
the earthquake in L'Aquila, where I lived with my family. Referring
to that period, I kindly acknowledge the public library of Codroipo
for the computer facilities, the ``Protezione Civile" of Codroipo
for the bureaucratic help, the colleagues who have sent me files and
who have substituted me in teaching,
   the Department of Mathematics and the office     ``Affari Sociali" of the      University
     ``La Sapienza".
%   the friends and relatives
 %   who have hosted me and my family. % I also thank for the warmness and affection received from many known and unknown persons.

%%%%%%%%%%%%%%%%%%%%%%%%%%%%%%%%%%
%%
%%      BIBLIOGRAPHY
%%
%%
%%%%%%%%%%%%%%%%%%%%%%%%%%%%%%%%%%%%%%%%%%%%%%%%%%%

\end{document}